\documentclass[12pt]{amsart}

 \usepackage{amscd, graphicx}
 \usepackage{tikz-cd}
 \usepackage{ifthen}    
 \usepackage{amssymb}   
 \usepackage{amstext}
 \usepackage{amsfonts}
 \usepackage{amsmath}   
 \usepackage{latexsym}  
 \usepackage{color}
 \usepackage{epsfig}
 \usepackage{pstricks, pst-node,pst-text,pst-tree}
 \usepackage[all]{xy}
 \usepackage{array}
 \usepackage{amsthm}
 \usepackage{eucal, mathrsfs} 
 \usepackage{hyperref}

\newcommand{\numberset}[1]{\ensuremath{\mathbb{#1}}}    
\newcommand{\C}{\numberset{C}}  
\newcommand{\R}{\numberset{R}}  
\newcommand{\Z}{\numberset{Z}}  
\newcommand{\Q}{\numberset{Q}}  
\newcommand{\PP}{\numberset{P}}  

\theoremstyle{definition}

\newtheorem{thm}{Theorem}[section]
\newtheorem{thm_int}{Theorem}

\newtheorem{prop}[thm]{Proposition}
\newtheorem{lem}[thm]{Lemma}
\newtheorem{cor}[thm]{Corollary}

\newtheorem{defi}[thm]{Definition}

\DeclareMathOperator{\id}{Id}
 
 \DeclareMathOperator{\Sl}{SL}
\DeclareMathOperator{\spn}{span}

\newboolean{mynotes}\setboolean{mynotes}{false}

\newcommand{\mycomments}[1]{
           \ifthenelse{\boolean{mynotes}}
                      {#1}{}
           }

\begin{document}

\title{On real Calabi-Yau threefolds twisted by a section}
\author{Diego Matessi}

\begin{abstract}
We study the mod $2$ cohomology of real Calabi-Yau threefolds given by real structures which preserve the torus fibrations constructed by Gross. We extend the results of Casta\~no-Bernard-Matessi and Arguz-Prince to the case of real structures twisted by a Lagrangian section. In particular we find exact sequences linking the cohomology of the real Calabi-Yau with the cohomology of the complex one. Applying SYZ mirror symmetry, we show that the connecting homomorphism is determined by a ``twisted squaring of divisors'' in the mirror Calabi-Yau, i.e.  by  $D \mapsto D^2 + DL$ where $D$ is a divisor in the mirror and $L$ is the divisor mirror to the twisting section. We use this to find an example of a connected $(M-2)$-real quintic threefold. 
\end{abstract}

\maketitle


\section{Introduction}

A real structure on a complex manifold $X$ is an anti-holomorphic involution $\iota: X \rightarrow X$,  as for example conjugation on an algebraic variety $X \subset \mathbb{CP}^n$ defined over $\R$. The real part of $X$ is the fixed point set of $\iota$, which we denote by $\Sigma$. Understanding the topology of $\Sigma$ is a notoriously difficult problem, see for instance Wilson's classical survey on Hilbert's sixteenth problem \cite{wilson:hilb}. A remarkable method to construct real hypersurfaces in toric varieties with controlled topology is Viro's patchworking technique \cite{viro, viro_patch, iten_mik_shu:trop}, which laid the foundations for the field of tropical geometry. Indeed patchworking allowed the construction of many interesting examples and counterexamples, especially in the case of curves and surfaces (\cite{iten:ragsdale, iten:t_surf}).  

One of the questions one may ask is about the relationship between the topologies of $\Sigma$ and $X$. For instance, a famous result is the Smith-Thom inequality relating the $\Z_2$ Betti numbers of $\Sigma$ and $X$:  
\[ \sum b_j(\Sigma, \Z_2) \leq \sum b_j(X, \Z_2). \]
When the equality is satisfied, then $\Sigma$ is said to be maximal, or an $M$-real hypersurface. We say it is of type $(M-k)$ if the difference between the two sums of Betti numbers is $2k$. There are many examples of maximal hypersurfaces in the case of curves and surfaces, but little is known in higher dimensions \cite{iten:t_surf}. Another problem is to find bounds on individual Betti numbers. For instance, a sharp bound on individual Betti numbers of real surfaces in $\mathbb{RP}^3$ is unknown in high degrees \cite{itenberg:trophom_betti}. 

One may investigate the same questions for real hypersurfaces constructed via patchworking. In this context, Itenberg \cite{itenberg:trophom_betti} conjectured that if $\Sigma$ is a hypersurface in $\mathbb{RP}^{n+1}$ constructed by primitive patchworking then 
 \[ b_q(\Sigma , \Z_2) \leq \begin{cases} h^{q,q}( X)  \ \text{if} \ q = n/2, \\ 
                                                          h^{q, n-q}( X) + 1 \ \text{otherwise}. 
                                                          \end{cases}. \]
This conjecture has been recently proved by Renaudineau and Shaw  \cite{ren_shaw:bound}, who actually proved a more general version for real hypersurfaces in toric varieties (see inequalities \eqref{bounds:bh}). In this paper we investigate similar questions but for real structures arising in a different context. 

\subsection{Lagrangian fibrations with real structures} Our goal is to generalize the results of Casta\~no-Bernard and Matessi \cite{CBM:fixpnt} and Arguz and Prince \cite{arg_princ:realCY} on the cohomology of real Calabi-Yau threefolds constructed via Lagrangian torus fibrations.   In this context $(X, \omega)$ is a $2n$-dimensional symplectic manifold, with symplectic form $\omega$, together with a Lagrangian torus fibration $f: X \rightarrow B$ onto a real $n$-dimensional manifold $B$. A compatible real structure is an anti-symplectic involution $\iota$, i.e. $\iota^*\omega = - \omega$, which preserves the fibres of the torus fibration. The real variety $\Sigma$ is the fixed point set of $\iota$. Let $\pi: \Sigma \rightarrow B$ be the restriction of $f$ to $\Sigma$. The general idea in  \cite{CBM:fixpnt, arg_princ:realCY}  and in this paper is to relate the cohomology with $\Z_2$ coefficients of $X$ and $\Sigma$ by comparing the Leray spectral sequences associated to $f$ and $\pi$.  

The torus fibrations which we consider in this paper are those constructed topologically by Gross in \cite{TMS} starting from the data of a three dimensional affine manifold with singularities $B$. It follows from \cite{Gross_Batirev} that one can construct these fibrations over affine manifolds with singularities associated to Calabi-Yau hypersurfaces or complete intersections in toric Fano varieties. It is expected, although not yet proved, that $X$ is homeomorphic to the corresponding Calabi-Yau. In the case of the quintic threefold in $\PP^4$, this has been proved by Gross in \cite{TMS} .  The main feature of Gross' fibrations is that they are built to naturally incorporate Strominger-Yau-Zaslow (SYZ) mirror symmetry at a  topological level. In fact there is a standard procedure to dualize the torus fibrations to obtain the mirror Calabi-Yau $\check X$ together with a fibration $\check f: \check X \rightarrow B$.  In \cite{CB-M} it was shown that Gross' fibrations could be made into Lagrangian fibrations with respect to a symplectic form extending the natural one existing on the union of smooth fibres.  These fibrations also come with a Lagrangian zero section $\sigma_0: B \rightarrow M$. 

A family of fibre preserving real structures on $X$ was constructed in \cite{CBMS}. We have the ``standard real structure'' which fixes the zero section. Denote the corresponding real Calabi-Yau by $\Sigma$. Then $\Sigma$ has at least two connected components,  one of them being the zero section, isomorphic to $B$. Given a Lagrangian section $\tau: B \rightarrow X$, one can ``twist" the standard real structure to get another real structure $\iota_{\tau}$. If $\tau$ is not the square of another section, then $\iota_{\tau}$ does not fix a section and therefore it is not standard. Let us denote by $\Sigma_{\tau}$ the corresponding real Calabi-Yau. The results in \cite{CBM:fixpnt} and \cite{arg_princ:realCY} concern the topology of the standard real Calabi-Yau. In this paper we generalize to the twisted case. 

\subsection{The Leray spectral sequence and mirror symmetry} The Leray spectral sequence of a Gross fibration  was investigated in \cite{splagI, splagII, TMS}.  Given some group of coefficients $G$, we have the sheaves $R^pf_{\ast}G$ on $B$, whose stalk at a point $b$ is the cohomology of the fibre $F_b$, i.e.  $H^p(F_b, G)$.  The second page of the Leray spectral sequence is given by $E^{q,p}_2 = H^q(B, R^pf_{\ast}G)$.  Mirror symmetry between $X$ and $\check  X$ implies the following isomorphism 
\[ H^q(B, R^pf_{\ast}G ) \cong H^q(B, R^{n-p}\check f_{\ast}G ). \]
Gross shows that for various choices of $G$ (e.g. $G = \Q$, $\Z$ or $\Z_p$) and with some assumptions on $B$, $X$ and $\check X$, the spectral sequence degenerates at the $E_2$ page. In this case, the cohomology of $X$ can be read off from the $E_2$ page. In particular the Hodge numbers of $X$ satisfy
\[ h^{p,q} (X) = \dim  H^q(B, R^pf_{\ast} \Q ).\]
This equality holds in higher dimensions and it has been proved in more generality in \cite{G-Siebert2007}. Notice that together with the above mirror symmetry isomorphism, this implies the famous relationship between the Hodge numbers of mirror Calabi-Yau manifolds $h^{p,q}(X) = h^{n-p, q}( \check X)$.

\subsection{Main results} Let $B_0$ be the set of regular values of $f$, so that for every $b \in B_0$, the fibre $F_b = f^{-1}(b)$ is a smooth $n$-dimensional torus. By the Arnold-Liouville theorem, $F_b$ is of the type $V/ \Lambda^*$ where $V$ is an affine space modeled on $T^{\ast}_bB_0$  and $\Lambda^* \cong \Z^n$ is an $n$-dimensional lattice in  $T^{\ast}_bB_0$. It follows that a compatible real structure $\iota$ on $X$, restricted to the fibre $F_b$, acts as reflection with respect to some point on $V$. In particular $\pi^{-1}(b) = \Sigma \cap F_b$ consists of $2^n$ points which have the structure of an $n$-dimensional affine space defined over $\Z_2$. In the case of Gross' fibrations, $\pi^{-1}(b)$ is finite for all $b \in B$. In particular, the Leray spectral sequence of $\pi$ is trivial: the cohomology of $\Sigma$ satisfies 
\[ H^{q}(\Sigma, \Z_2) \cong H^{q}(B, \pi_{\ast} \Z_2). \]

Our results considers the case when $X$ is a Calabi-Yau threefold, i.e. $n=3$. The first result is the following. 

\begin{thm_int} 
Let $\tau$ be a Lagrangian section of $f: X \rightarrow B$ and $\iota_{\tau}$ the associated real structure. There exist sheaves $\mathcal L^{1}_{\tau}$ and $\mathcal L^{2}_{\tau}$ over $B$ and a short exact sequence 
\[
            0 \longrightarrow \mathcal L^{1}_{\tau}  \longrightarrow \pi_{\tau_\ast} \Z_2 \longrightarrow \mathcal L^{2}_{\tau} \longrightarrow 0,
\]
such that $\mathcal L^{1}_{\tau}$ and $\mathcal L^{2}_{\tau}$ are related to the topology of $X$ by the following short exact sequences
\[
            0 \longrightarrow \Z_2  \longrightarrow \mathcal L^1_{\tau} \longrightarrow  R^1 f_{\ast} \Z_2 \longrightarrow 0,
\]
\[
            0 \longrightarrow R^2 f_{\ast} \Z_2  \longrightarrow \mathcal L^2_{\tau} \longrightarrow \Z_2 \longrightarrow 0.
\]
\end{thm_int} 

Notice that at a regular value $b \in B_0$, we have 
\[ (\pi_{\tau_\ast} \Z_2)_{b} = \operatorname{Maps}(\pi^{-1}(b), \Z_2). \] The sheaf $\mathcal L^1_{\tau}$ in $b$ coincides with the affine maps. This also explains the second sequence, which is the usual splitting of affine functions as the sum of a constant function and a linear function.  In the case of the standard real structure, the first sequence coincides with the one found in \cite{CBM:fixpnt}. Indeed in this case $\pi^{-1}(b)$ is naturally a vector space, since the zero section defines an origin. This implies that the second and third sequence are both split.  

The first sequence gives a long exact sequence in cohomology  which computes the $\Z_2$ cohomology  of $\Sigma_{\tau}$. In particular we have the connecting homomorphism 
\[ \beta: H^{1}(B,  \mathcal L^{2}_{\tau} )  \rightarrow H^{2}(B,  \mathcal L^{1}_{\tau} ). \]
By composing $\beta$ with the morphisms from the second and third sequence (see diagram \eqref{beta}) we get the homomorphism
\[ \beta':  H^{1}(B, R^2  f_{\ast} \Z_2) \rightarrow H^{2}(B, R^1 f_{\ast} \Z_2). \]
It was shown in  \cite{CBM:fixpnt} that in the case $B$ is a homology $\Z_2$-sphere and $H^{1}(X, \Z_2) = H^{1}(\check X, \Z_2) =0$, then the untwisted $\Sigma$ has exactly two connected components. Under the same hypothesis, we prove here that in the twisted case, $\Sigma_{\tau}$ is connected. In both cases the cohomology of $\Sigma_{\tau}$ is uniquely determined by $\beta'$.  As a corollary of this construction we also get that if in addition the integral cohomologies of $X$ and $\check X$ have no torsion, then the Betti numbers of $\Sigma_{\tau}$ satisfy the same bounds as those proved by Renaudineau-Shaw (inequalities \eqref{bounds:bh}). Indeed, in the twisted case the bound is stronger: a twisted $\Sigma_{\tau}$ can be at most of type $(M-2)$ and this happens if and only if $\beta'$ is the zero map. 

If we apply the mirror symmetry isomorphism we can view $\beta'$ as a map $\beta': H^{1}(B, R^1 \check f_{\ast} \Z_2) \rightarrow H^{2}(B, R^2 \check f_{\ast} \Z_2)$ on the cohomology of the mirror $\check X$. We now have that a Lagrangian section $\tau$ can be naturally viewed as a class $\tau \in  H^{1}(B, R^2  f_{\ast} \Z_2)$. If we apply mirror symmetry to $\tau$, we get an element $L_{\tau} \in  H^{1}(B, R^1 \check f_{\ast} \Z_2)$. Notice that the latter group can be interpreted as a $\Z_2$ version of the Picard group of $\check X$ and in particular $L_{\tau}$ can be viewed as the line bundle mirror to the section $\tau$. 

We can now state the main theorem of this paper.

\begin{thm_int}
The map $\beta': H^{1}(B, R^1 \check f_{\ast} \Z_2) \rightarrow H^{2}(B, R^2 \check f_{\ast} \Z_2)$ coincides with the map 
\[ \begin{split} 
            S_{\tau} : H^{1}(B, R^1 \check f_{\ast} \Z_2) & \longrightarrow H^{2}(B, R^2 \check f_{\ast} \Z_2) \\
                                                                                            D  & \longmapsto D^2 +  DL_{\tau}
   \end{split}
            \]
\end{thm_int}

This theorem generalizes the main result in \cite{arg_princ:realCY},  which proves the untwisted case, i.e. when $L_{\tau} = 0$. 

As an application we find a connected $(M-2)$-real quintic. We use the torus fibration on a quintic in $\PP^4$ constructed by Gross in \cite{TMS}. Then, on the mirror quintic $\check X$, we have that $H^{1}(B, R^1 \check f_{\ast} \Z_2) \cong H^2(\check X, \Z_2)$, which coincides with the Picard group$\mod 2$. To find an $(M-2)$-real quintic it is then enough to find an $L \in  H^2(\check X, \Z_2)$ such that $D^2 + DL = 0$ for all $D \in H^2(\check X, \Z_2)$. The real quintic will then be $\Sigma_{\tau}$, such that $L = L_{\tau}$. We find such an $L$ by using the explicit description of the triple intersection form on $\check X$ given in \cite{TMS}.  Arguz and Prince have computed the Betti numbers of the untwisted real quintic $\Sigma$, obtaining $b_1(\Sigma) = 29$. In particular $\Sigma$ is far from being maximal, therefore none of the real quintics constructed with this method is maximal. 

\subsection{Structure of the paper} In Section \ref{lagFib} we review the necessary background on Gross' construction of torus fibrations, topological mirror symmetry and the Leray spectral sequence. In Section \ref{real_str} we recall the setup in \cite{CBMS} where the standard and twisted real structures are defined and we discuss the results in  \cite{CBM:fixpnt} and \cite{arg_princ:realCY} for the standard real structure.   In Section \ref{seseq} we prove Theorem 1 (i.e. Theorem \ref{sheaf_seq}).  In Section \ref{cnnctng} we prove Theorem 2 (i.e. Theorem \ref{beta_MS}).  In Section \ref{applications} we prove some consequences, such as connectedness and bounds on the Betti numbers. In Section \ref{rs:sq} we discuss the relationship between our short exact sequences and the spectral sequence constructed Renaudineau and Shaw. In Section \ref{m-2:quintic} we explain the construction of the connected $(M-2)$ real quintic. 

\subsection*{Acknowledgments} 
I wish to thank Arthur Renaudineau for explaining me most of Section \ref{rs:sq}, Mark Gross for a useful discussion, H\"ulya Arguz and Thomas Prince for explaining me their work. I was partially supported by the national research project ``Moduli and Lie theory'' (PRIN 2017) and by a travel grant from the INDAM research group GNSAGA.

\section{Lagrangian fibrations and mirror symmetry} \label{lagFib}

We explain the construction of Lagrangian torus fibrations starting from the data of an integral affine manifold with singularities. The topological construction was done by Gross \cite{TMS} for the $3$-fold case and an extension to all dimensions was announced by Ruddat and Zharkov \cite{rud_zha:torus_fib}. It was shown in \cite{CB-M} that a variant of Gross' topological fibrations are indeed Lagrangian with respect to a symplectic form on $X$, induced by the integral affine structure on $B$. 

\begin{defi} An integral affine manifold with singularities is a triple $(B, \Delta, \mathcal A)$ where $B$ is an $n$-dimensional topological manifold; $\Delta \subset B$ a closed, codimension $2$ subset and $\mathcal A$ a maximal atlas on $B_0 = B - \Delta$ whose change of coordinate maps are in $\R^n \rtimes \Sl (\Z,n)$. The set $\Delta$ is called the discriminant locus. 
\end{defi}

Given $B_0 = B - \Delta$, we denote by $j: B_0 \rightarrow B$ the inclusion.  The cotangent bundle $T^*B_0$ carries the standard symplectic form, moreover we have the lattice 
\[ \Lambda^* = \spn_{\Z} \langle dx_1, \ldots, dx_n \rangle, \]
where $(x_1, \ldots, x_n)$ are local integral affine coordinates. This defines the symplectic manifold 
\[ X_0 = T^*B_0 / \Lambda^* \] 
together with the Lagrangian torus fibration $f_0: X_0 \rightarrow B_0$ given by the standard projections. A (partial) compactification of $X_0$ is given by a $2n$-dimensional manifold $X$ together with map $f: X \rightarrow B$ and a commutative diagram
\begin{equation} \label{compact}   \begin{tikzcd}
  X_0 \arrow[r] \arrow[d, "f_0"]
    & X \arrow[d, "f"] \\
  B_0 \arrow[r, "j"]
&  B 
\end{tikzcd}
\end{equation}
where the top arrow is a homeomorphism onto its image. In dimension $n=3$, Gross shows that under certain hypothesis on the set $\Delta$ and the affine structure around it, such a compactification can be carried out topologically in a canonical way.  In the same dimension and with the same hypothesis, Casta\~no-Bernard and Matessi \cite{CB-M} prove that, after a small thickening of $\Delta$, $X$ has a symplectic structure such that the inclusion $X_0 \rightarrow X$ is a symplectomorphism and $f$ is a Lagrangian fibration . 

The hypothesis on $\Delta$ and on the affine structure is that they are locally isomorphic to certain prescribed local models. When these hypothesis are satisfied we will say that $B$ is {\it simple}. We describe below the local models in dimension $2$ and $3$. In higher dimensions there is a longer list. The models are characterized by two key properties of the monodromy of $\Lambda^*$ around $\Delta$ which are called simplicity and positivity, see \cite{G-Siebert2003} for further details. 

\subsection{Dimension two: focus-focus points.} In dimension two we ask for $\Delta$ to be a discrete set of points with an affine structure around them locally isomorphic to the one depicted in Figure \ref{focus_focus}. Such singular points are called focus-focus points. In Figure \ref{focus_focus} the two simplices on the left are glued to the simplices on the right via integral affine transformations. The (red) cross is the point in $\Delta$.

\begin{figure}[!ht] 
\begin{center}
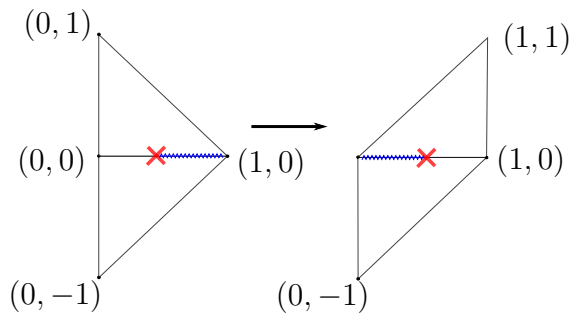
\caption{Charts around a focus-focus point}
\label{focus_focus}
\end{center}
\end{figure}

\subsection{The affine quartic} This is a global compact example. Take the simplex $P$ in $\R^3$ whose vertices are 
\[ \begin{split}
   & p_0 = (-1, -1, -1),  \quad p_1= (3, -1,-1),  \\ 
   & p_2=(-1,3,-1),    \quad p_3=(-1,-1, 3) 
   \end{split} 
 \]
We take $B = \partial P$. Each edge of $P$ contains five integral points and is subdivided by these into four segments. Define $\Delta$ to be the set of midpoints of these segments. In total $\Delta$ consists of $24$ points. We can define charts on $B$ as follows. We have four charts consisting of the interior of each $2$-face of $P$ together with their natural integral affine structure. For every integral point $q$ on some edge consider a neighborhood $U_q \subset B- \Delta$. Define a chart by the projection $U_q \rightarrow \R^3/\R q$, where $\R q$ is the line generated by $q$. By choosing these neighborhoods so that they cover $B_0$, we obtain an integral affine structure on $B_0$. It is not hard to show that all $24$ points are focus-focus.  This example is called the affine quartic because it is the affine structure associated to a toric degeneration of a quartic in $\PP^3$. 

\subsection{Dimension three: positive and negative vertices} In dimension three, $\Delta$ must be a trivalent graph. We have three local models. One for a generic point along an edge of $\Delta$ and two local models for vertices, which can be either of positive or negative type. 

The affine structure along an edge of $\Delta$ has the following description. Take the focus-focus $2$-dimensional model, denote it by $(B_{ff}, p)$, where $p$ is the focus-focus singular point. Then along the interior of an edge of $\Delta$ we want the affine structure to be locally isomorphic to $B_{ff} \times \R$, where now $\Delta = \{p\} \times \R$.

 A {\it negative vertex} is depicted in Figure \ref{negative}. Here $B$ is the union of two standard simplices and $\Delta$ is the trivalent graph (with just one vertex) depicted in red inside the common face. Figure \ref{negative} depicts the affine structure on $B_0$. It has three charts, one for each vertex of the common face. In the figure the shaded regions are not part of the charts. 

\begin{figure}[!ht] 
\begin{center}
\begingroup%
  \makeatletter%
  \providecommand\color[2][]{%
    \errmessage{(Inkscape) Color is used for the text in Inkscape, but the package 'color.sty' is not loaded}%
    \renewcommand\color[2][]{}%
  }%
  \providecommand\transparent[1]{%
    \errmessage{(Inkscape) Transparency is used (non-zero) for the text in Inkscape, but the package 'transparent.sty' is not loaded}%
    \renewcommand\transparent[1]{}%
  }%
  \providecommand\rotatebox[2]{#2}%
  \newcommand*\fsize{\dimexpr\f@size pt\relax}%
  \newcommand*\lineheight[1]{\fontsize{\fsize}{#1\fsize}\selectfont}%
  \ifx\svgwidth\undefined%
    \setlength{\unitlength}{177.37744904bp}%
    \ifx\svgscale\undefined%
      \relax%
    \else%
      \setlength{\unitlength}{\unitlength * \real{\svgscale}}%
    \fi%
  \else%
    \setlength{\unitlength}{\svgwidth}%
  \fi%
  \global\let\svgwidth\undefined%
  \global\let\svgscale\undefined%
  \makeatother%
  \begin{picture}(1,0.91659668)%
    \lineheight{1}%
    \setlength\tabcolsep{0pt}%
    \put(0,0){\includegraphics[width=\unitlength,page=1]{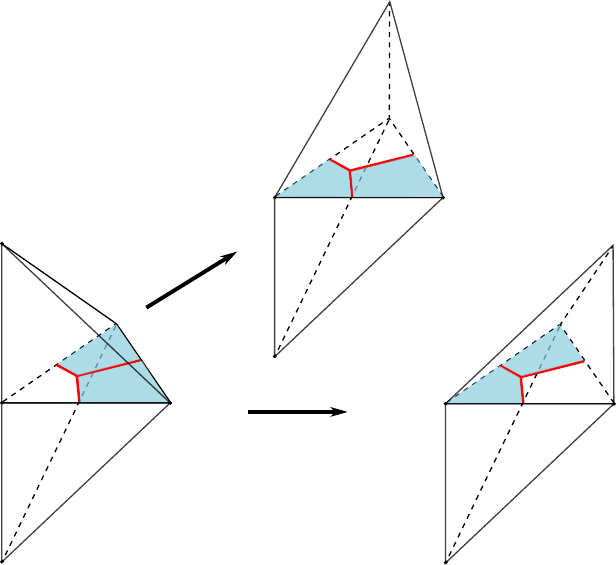}}%
    \put(0.2370831,0.51061475){\color[rgb]{0,0,0}\makebox(0,0)[lt]{\lineheight{1.25}\smash{\begin{tabular}[t]{l}$\Phi_1$\end{tabular}}}}%
    \put(0.44592307,0.17731459){\color[rgb]{0,0,0}\makebox(0,0)[lt]{\lineheight{1.25}\smash{\begin{tabular}[t]{l}$\Phi_2$\end{tabular}}}}%
  \end{picture}%
\endgroup%

\caption{Charts near a negative vertex}
\label{negative}
\end{center}
\end{figure}

The map $\Phi_1$ is the identity on the bottom simplex and on the top simplex it is  the linear map given by the matrix 
\[ \begin{pmatrix}
          1 & 0 & 0 \\
           0 & 1 & 1 \\
           0 & 0 & 1
    \end{pmatrix}. \] 
The map $\Phi_2$ is the identity on the bottom simplex and on the top simplex it is  the linear map given by the matrix 
\[ \begin{pmatrix}
          1 & 0 & 1 \\
           0 & 1 & 0 \\
           0 & 0 & 1
    \end{pmatrix}.  \] 

For the {\it positive vertex}, we let $B = \R \times \R^2 $ and take $\Delta$ inside $ \{0\} \times \R^2$ given by the set
\[ \{ y=0 , \ x \geq 0 \} \cup \{ x=0, \ y \geq 0 \} \cup \{ x=y, \ x \leq 0 \}. \]
Define the sets 
\[ R^+ = \R_{\geq 0} \times \Delta, \quad  R^- = \R_{\leq 0} \times \Delta \]
The affine structure on $B_0$ has two charts. The open sets are
\[ U_1 = (\R \times \R^2) - R^-, \quad U_2 =  (\R \times \R^2) - R^+.\]
On $\R^2$ define the piecewise linear function 
\[ \nu(x,y) = \min \{0,x,y \} \]
Define the coordinate map $\phi_1$ on $U_1$ to be the identity and the coordinate map $\phi_2$ on $U_2$ to be
\[ \phi_2(t,x,y) = (t + \nu(x,y), x,y). \]

\subsection{The affine quintic} \label{aff_quint}
This example is similar to the affine quartic, but one dimension higher. Take the simplex $P$ in $\R^4$ with vertices
\[ \begin{split}
    p_0 =  (-1, -1, & -1, -1),  \quad p_1= (4,  -1, -1,  -1 ),   \quad p_2=(-1,4,-1, -1), \\ 
    & p_3=(-1,-1, 4, -1), \quad  p_4 = (-1, -1,-1, 4).
   \end{split} 
 \]
 \begin{figure}[!ht] 
\begin{center}
\includegraphics{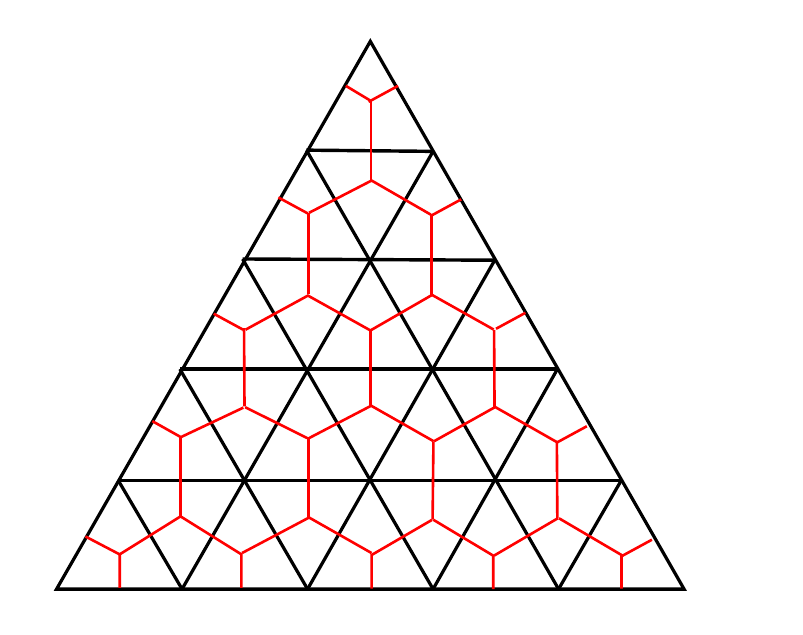}
\caption{Discriminant of an affine quintic}
\label{aff_quintic}
\end{center}
\end{figure}
 
Let $B = \partial P$. Inside every $2$-face of $P$, consider the honeycomb (red) graph depicted in Figure \ref{aff_quintic}. Define $\Delta$ to be union of such graphs over all $2$-faces of $P$.  The interior of each $2$-face contains $25$ trivalent vertices. There are also $5$ trivalent vertices in the interior of each edge. These are the points where the honeycomb graph intersects an edge, indeed each edge is contained in exactly three $2$-faces. We can define charts on $B$ as follows. We have obvious charts consisting of the interior of each $3$-face of $P$. For every integral point $q$ in a $2$-face, consider a neighborhood $U_q \subset B- \Delta$. Define a chart by the projection $U_q \rightarrow \R^4/\R q$, where $\R q$ is the line generated by $q$. By choosing these neighborhoods so that they cover $B_0$, we obtain an integral affine structure. It can be shown that vertices in the interior of $2$-faces are of negative type and vertices in the interior of edges are of positive type.  This example is called the affine quintic because it is the affine structure associated to a toric degeneration of a quintic in $\PP^4$. More examples of similar affine manifolds with singularities associated to toric degenerations of Calabi-Yau complete intersections in Fano toric varieties are constructed in \cite{Gross_Batirev}.

\subsection{Singular fibres}  The compactification in diagram \eqref{compact} is obtained by gluing suitable singular fibres over $\Delta$. For instance, in dimension $2$, the singular fibre over a focus-focus point is a once pinched torus.  In dimension three the singular fibre over a point in the interior of an edge of $\Delta$ is $F \times S^1$ where $F$ is a once pinched torus. The fibre over a positive vertex is obtained by considering a three torus $T^2 \times S^1$, where $T^2$ is a two torus, and collapsing a two torus $T^2 \times \{p\}$ to a point. The singular fibre over a negative vertex is more complicated, we refer to \cite{TMS} or \cite{CB-M} for the Lagrangian models.  In the case of the affine quartic the compactified manifold $X$ is homeomorphic to a K3 surface, i.e. to a quartic and the affine quintic is homeomorphic to a quintic Calabi-Yau, as proved by Gross in \cite{TMS}. It is expected that when $X$ is constructed from affine manifolds with singularities associated to toric degenerations of Calabi-Yau complete intersections in Fano toric varieties as in \cite{Gross_Batirev}, then it is homeomorphic to the given Calabi-Yau.

\subsection{Topological mirror symmetry}
In \cite{TMS} Gross constructs the topological mirror $\check X$ of $X$. Given the lattice $\Lambda \subset TB_0$, dual to $\Lambda^*$, we can form 
\[ \check X_0 = TB_0 / \Lambda \]
together with projection $\check f_0: \check X_0 \rightarrow B_0$. Gross proved that when $B$ is simple, also $\check X_0$ can be compactified to a manifold $\check X$ with a map $\check f: \check X \rightarrow B$ extending $\check f_0$. Indeed, in dimension $3$, the positive fibres in $X$ must be replaced by negative fibres in $\check X$ and viceversa. 

Since the tangent bundle does not have a natural symplectic structure, to construct a Lagrangian fibration on $\check X$ one needs the additional data of a potential $\phi$. This is a multivalued strictly convex function which can be used to define a symplectic form on $TB_0$ or, equivalently, to define a mirror affine structure on $B_0$ via a Legendre transform. For the purpose of this paper it will be enough to consider the mirror $\check X$ as the topological manifold obtained from $TB_0$. 

\subsection{The Leray spectral sequence} \label{Leray}
The cohomology of $X$ can be computed by the Leray spectral sequence associated to the map $f: X \rightarrow B$. Recall that given a group $G$ we denote by $R^pf_*G$ the sheaf on $B$ associated to the presheaf $U \mapsto H^p(f^{-1}(U), G)$. The fibration is called {\it $G$-simple} if
\[ j_{\ast}R^pf_{0_\ast}G = R^pf_{\ast}G \]
This essentially means that the cohomology of the singular fibres is determined by the local monodromy of $\Lambda^*\otimes G$. Gross proves that the fibrations constructed above (i.e. from a simple $B$) are $G$-simple for $G = \Z, \Q$ and $\Z_n$. 

The $E_2$ page is given by the cohomology groups $H^q(B,  R^pf_{\ast}G)$. Since the fibres are connected, we have that 
\begin{equation}  \label{R0}
      R^0f_{\ast}G \cong G.
\end{equation}
Let us now consider $G = \Z$. The fact that transition maps of the affine structure are in $\R^n \rtimes \Sl (\Z,n)$ implies that the fibres are oriented, in particular  
\begin{equation}  \label{Rn}
      R^nf_{\ast} \Z \cong \Z. 
\end{equation}
This is equivalent  to the fact that $B_0$ has a global integral volume form.  If $b \in B_0$, we have that 
\begin{equation} \label{Rlam}
 (R^p f_{\ast} \Z)_b = \bigwedge^p \Lambda_b. 
\end{equation}
On the other hand, if we consider the Leray spectral sequence for the mirror we have 
\begin{equation} \label{Rcheck}
   (R^p \check f_{\ast} \Z)_b = \bigwedge^p \Lambda^*_b.
\end{equation}
By contraction with the global volume form (or equivalently by Poincar\'e duality on the fibres) we have the natural isomorphism
\[  R^p f_{0_\ast} \Z  \cong R^{n-p} \check f_{0_\ast} \Z. \]
This extends to an isomorphism 
\begin{equation}  \label{Riso}  
              R^p f_{\ast} \Z  \cong R^{n-p} \check f_{\ast} \Z 
\end{equation}
by $\Z$-simplicity. 

We now consider $n=3$ and $G =\Q$. With the additional assumptions that $B$ is a $\Q$-homology sphere and that $b_1(X) = b_1(\check X) = 0$ it can be shown that the $E_2$ page for $X$ looks as follows
\[
\begin{array}{cccc}
\Q &           0                         &                 0                  & \Q \\
0   & H^1(B, R^2f_{\ast} \Q) & H^2(B, R^2f_{\ast} \Q) & 0  \\
0   & H^1(B, R^1f_{\ast} \Q) & H^2(B, R^1f_{\ast} \Q) & 0  \\
\Q &           0                         &                 0                  & \Q 
\end{array}
\]
The bottom and top rows follow from \eqref{R0} and \eqref{Rn} and the assumption that $B$ is a $\Q$-homology sphere. The vanishing of $H^0(B,  R^1f_{\ast}\Q)$ and $H^3(B,  R^2f_{\ast}\Q)$ follow from the assumption $b_1(X) = b_5(X) =0$. The vanishing of $H^0(B,  R^2f_{\ast}\Q)$ and $H^3(B,  R^1f_{\ast}\Q)$ follow from \eqref{Riso} and the assumption $b_1(\check X) = b_5(\check X) =0$. 

Gross \cite{splagI, splagII, TMS} proves that, with the given hypothesis, the spectral sequence degenerates at the $E_2$ page. In particular we have
\[ H^2(X, \Q) \cong H^1(B, R^1f_{\ast} \Q) \cong H^2(B, R^2f_{\ast} \Q) \cong H^4(X, \Q) \]
and similarly for $\check X$. Using \eqref{Riso} we also have 
\[ H^1(B, R^2f_{\ast} \Q) \cong H^1(B, R^1 \check f_{\ast} \Q) \cong H^2(B, R^2 \check f_{\ast} \Q) \cong  H^2(B, R^1f_{\ast} \Q) \]

If $X$ and $\check X$ are Calabi-Yau manifolds we have that the Hodge numbers satisfy
\[ h^{1,1}(X) = \dim H^1(B, R^1f_{\ast} \Q) = \dim H^1(B, R^2 \check f_{\ast} \Q) = h^{1,2}(\check X) \]
In particular, we have the celebrated mirror symmetry of the hodge diamonds of $X$ and $\check X$. 

In this paper we will be concerned with  cohomology with $\Z_2$ coefficients. Also in this case the spectral sequence degenerates at the $E_2$ page and if we assume that $B$ is a $\Z_2$-cohomology sphere and $H^1(X, \Z_2) \cong H^1(\check X, \Z_2) \cong 0$, the $E_2$ page becomes 
\[
\begin{array}{cccc}
\Z_2 &           0                         &                 0                  & \Z_2 \\
0   & H^1(B, R^2f_{\ast} \Z_2) & H^2(B, R^2f_{\ast} \Z_2) & 0  \\
0   & H^1(B, R^1f_{\ast} \Z_2) & H^2(B, R^1f_{\ast} \Z_2) & 0  \\
\Z_2 &           0                         &                 0                  & \Z_2 
\end{array}
\]
Again we have 
\[ 
\begin{split}
 H^2(X, \Z_2) & \cong H^1(B, R^1f_{\ast} \Z_2), \\
 H^1(B, R^2 f_{\ast} \Z_2) & \cong H^2(B, R^1 f_{\ast} \Z_2). \\
 \end{split}
\]
In particular the $\Z_2$-Betti numbers of $X$ satisfy 
\[ 
\begin{split}
 b_2(X, \Z_2) & = \dim H^1(B, R^1f_{\ast} \Z_2) \\
 b_3(X, \Z_2) & = 2 + 2 \dim H^1(B, R^2 f_{\ast} \Z_2) 
 \end{split}
\]
The relation between the hodge numbers of $X$ and the groups $H^q(B, R^p f_{\ast} \Z_2)$ depends on whether $H^{p+q}(X, \Z)$ has torsion. 

\section{Real structures} \label{real_str}

\subsection{The standard real structure} Consider the involution $\iota_0: X_0 \rightarrow X_0$ induced from the map $\alpha \mapsto -\alpha$ on the fibres of the cotangent bundle of $B_0$. It is clearly anti-symplectic in the sense that on the symplectic form $\omega$ it acts as $\iota_0^*\omega = - \omega$. It was proved in \cite{CBMS} that $\iota_0$ extends to a smooth fibre preserving anti-symplectic involution $\iota: X \rightarrow X$. It is clear that $\iota$ fixes the zero section. We call $\iota$ the standard real structure.  We denote by $\Sigma$ the fixed point set of $\iota$, which we can think as the real part of $X$. The zero section is a connected component of $\Sigma$.  We also denote by $\pi: \Sigma \rightarrow B$ the restriction of $f$, i.e. $\pi = f|_{\Sigma}$. Notice that $\pi$ is generically a $2^n$ to $1$ covering. 

\subsection{Twisted real structures} \label{twist_real} In \cite{CBMS} we constructed real structures which can be viewed as a twist of $\iota$ by a Lagrangian section.  Let $\tau: B \rightarrow X$ be a Lagrangian section. Consider on $X_0$ the translation by $\tau$, i.e. the map which on the fibres acts by $\alpha \mapsto \alpha + \tau$. It was shown in \cite{CBMS} that this map extends smoothly to a fibre preserving symplectomorphism of $X$, for simplicity we continue to denote it by $\tau$.  Now assume that $\tau$ is not a square, i.e. that there does not exist another section $\tau'$ such that $\tau = 2\tau'$.  Define 
\[ \iota_{\tau} = \iota \circ \tau \]
Clearly $\iota_\tau$ is a fibre preserving antisymplectic map. To prove that it is an involution, consider the map $\iota \circ \tau \circ \iota$. It is a fibre preserving symplectomorphism. As such, it must be given by the translation by a section. Indeed it is easy to show that 
\[ \iota \circ \tau \circ \iota = - \tau. \]
In particular 
\[ \iota_\tau ^2 = (\iota \circ \tau) \circ (\iota \circ \tau) = (\iota \circ \tau \circ \iota) \circ \tau = (-\tau) \circ (\tau) = \id_{X}. \]
Therefore $\iota_{\tau}$ is an involution. 

The fact that $\tau$ is not a square implies that $\iota_{\tau}$ does not fix any section. We call $\iota_{\tau}$ a {\it twisted real structure}, where $\tau$ is the twist. We denote by $\Sigma_{\tau}$ the fixed point set of $\iota_{\tau}$.  We also denote by $\pi_{\tau}: \Sigma_{\tau} \rightarrow B$ the map given by the restriction of $f$. Also in this case $\pi_{\tau}$ is generically an $2^n$ to $1$ covering. 

\subsection{A long exact sequence} Let us restrict to the three dimensional case $n=3$. We describe the long exact sequence which was found in \cite{CBM:fixpnt} relating the $\Z_2$-cohomology of $\Sigma$ with the cohomology of $X$. Consider the projection $\pi: \Sigma \rightarrow B$.  In particular, since the preimage by $\pi$ of a point in $B$ is finite, we have that the Leray spectral sequence of $\pi$ is quite simple and 
\[ H^{q}(\Sigma, G) \cong H^q(B, \pi_{\ast} G) \]
We restrict to the case $G=\Z_2$. In \cite{CBM:fixpnt} the following result is proved
\begin{lem} There exists a short exact sequence of sheaves on $B$:
\begin{equation} \label{long_seq}
            0 \longrightarrow R^1 f_{\ast} \Z_2 \oplus \Z_2 \oplus \Z_2  \longrightarrow \pi_{\ast} \Z_2 \longrightarrow R^2 f_{\ast} \Z_2 \longrightarrow 0
 \end{equation}
\end{lem}
\begin{proof} Let us consider the same sheaves restricted to $B_0$. Notice that for every $b \in B_0$, we have 
 \[ \begin{split} \pi^{-1}(b) &= \tfrac{1}{2} \Lambda^* \mod \Lambda^*  \\
                               \      & \cong \Lambda^* \otimes \Z_2
     \end{split} 
  \]
 In particular it is a subgroup of the fibre $f^{-1}(b)$ isomorphic to $\Z_2^{3}$. On the other hand the stalk of $\pi_{\ast} \Z_2$ at $b$ is canonically identified with 
 \[ (\pi_{\ast} \Z_2)_b = \operatorname{Maps} ( \pi^{-1}(b), \Z_2). \]
  Moreover we have isomorphisms \eqref{Rlam} for $\Z_2$ coefficients
\begin{equation} \label{R2} 
  \begin{split}
        & (R^1 f_{\ast} \Z_2)_b = \Lambda \otimes \Z_2,  \\ 
        & (R^2 f_{\ast} \Z_2)_b \cong \bigwedge^2 \Lambda  \otimes \Z_2.
  \end{split}
\end{equation}
Notice that $(R^1 f_{\ast} \Z_2)_b$ is the space of linear maps from $\pi^{-1}(b)$ to $\Z_2$, so it injects in  $(\pi_{\ast} \Z_2)_b$. Concerning the two $\Z_2$ summands in the lefthand side of the sequence we have
\[ \Z_2 \oplus \Z_2 = \langle 1 \rangle \oplus \langle \delta_0 \rangle \]
where $1$ is the constant map equal to $1$ and $\delta_0$ is the delta function at $0$ (i.e. the map which maps $0 \in \pi^{-1}(b)$ to $1$ and everything else to $0$). 

From \eqref{Rcheck} and $\eqref{Riso}$ we have the isomorphism
\begin{equation} \label{R2-1}
      (R^2 f_{\ast} \Z_2)_b \cong \Lambda ^* \otimes \Z_2 \cong \pi^{-1}(b).
 \end{equation}
One can identify $\pi^{-1}(b)$ with the quotient of the first two groups of the above sequence by identifying a non zero point $y \in \pi^{-1}(b)$ with the class (in the quotient) of the map $\delta_{y}$, the delta function at $y$. 
\end{proof}

Let us assume that $B$ is a homology $\Z_2$ sphere. Then the short exact sequence induces the long exact sequence in cohomology :

{\small \begin{equation}  \label{connecting}
\begin{split}
          0  \longrightarrow & H^{0}(B, R^1 f_{\ast} \Z_2) \oplus (\Z_2)^2 \longrightarrow H^{0}(\Sigma, \Z^2)   \longrightarrow   H^{0}(B, R^2 f_{\ast} \Z_2) \longrightarrow \\
          \ & \longrightarrow H^{1}(B, R^1 f_{\ast} \Z_2) \longrightarrow H^{1}(\Sigma, \Z^2)   \longrightarrow   H^{1}(B, R^2 f_{\ast} \Z_2) \stackrel{\beta}{\longrightarrow}  \\
          \ &  \longrightarrow H^{2}(B, R^1 f_{\ast} \Z_2)  \longrightarrow  H^{2}(\Sigma, \Z_2) \longrightarrow H^{2}(B, R^2 f_{\ast} \Z_2) \longrightarrow \ldots 
\end{split} 
\end{equation}} 

The map  $\beta:  H^{1}(B, R^2 f_{\ast} \Z_2) \rightarrow H^{2}(B, R^1 f_{\ast} \Z_2)$ is the connecting homomorphism and it essentially determines the cohomology of $\Sigma$.

In particular the following corollaries follow from the properties of the Leray spectral sequence described in \S \ref{Leray}. 

\begin{cor} \label{cor1_seq} If $B$ is a cohomology $\Z_2$-sphere and $H^1(X, \Z_2)$ and $H^1(\check X, \Z_2) $ are both zero, then $\Sigma$ has two connected components and the long exact sequence \eqref{connecting} splits as
\[ \begin{split} 0 & \rightarrow H^{1}(B, R^1 f_{\ast} \Z_2) \rightarrow H^{1}(\Sigma, \Z_2)   \rightarrow   H^{1}(B, R^2 f_{\ast} \Z_2) \stackrel{\beta}{\rightarrow} H^{2}(B, R^1 f_{\ast} \Z_2) \\
 & \rightarrow  H^{2}(\Sigma, \Z_2) \rightarrow H^{2}(B, R^2 f_{\ast} \Z_2) \rightarrow 0 
\end{split}\]
\end{cor}

\begin{cor} \label{cor2_seq}  Under the same hypothesis of Corollary \ref{cor1_seq}, with the additional assumption that $X$ is a Calabi-Yau variety and the cohomologies of $X$ and $\check X$ have no $2$-torsion, we have that the $\Z_2$ Betti numbers of $\Sigma$ satisfy
\[ b_q(\Sigma, \Z_2) \leq h^{q, 3-q}(X) + h^{q,q}(X). \]
Indeed the hypothesis and the properties of the Leray spectral sequence imply 
\[ \dim H^{p}(B, R^q f_{\ast} \Z_2) = h^{p,q}(X).  \]

These inequalities  coincide with those proved by Renaudineau-Shaw \cite{ren_shaw:bound} for any real hypersurface arising from primitive patchworking in a toric variety. Notice however that in our case $X$ is not necessarily a hypersurface (see for instance the case of Schoen's Calabi-Yau, \cite{CBM:fixpnt} and \cite{ar_pr:schoen}). 

\end{cor}

\subsection{Mirror symmetry} 
Using the isomorphism \eqref{Riso}, we can interpret the connecting homomorphism $\beta$ in \eqref{connecting} as a map on the cohomology of the mirror $\check X$:
\[  \beta:  H^{1}(B, R^1 \check f_{\ast} \Z_2) \rightarrow H^{2}(B, R^2 \check f_{\ast} \Z_2)\]
In \cite{arg_princ:realCY}, Arguz and Prince proved the following remarkable result
\begin{thm} \label{arguz_prince}
The connecting homomorphism $\beta$ in the long exact sequence \eqref{connecting} coincides with the squaring map 
\[ \begin{split} 
            \operatorname{Sq}: H^{1}(B, R^1 \check f_{\ast} \Z_2) & \longrightarrow H^{2}(B, R^2 \check f_{\ast} \Z_2) \\
                                                                                            D  & \longmapsto D^2
   \end{split}
            \]
\end{thm} 

Notice that if $B$ is a $\Z_2$ homology sphere and $H^1(X, \Z_2) = H^1(\check X, \Z_2) = 0$, then $H^{p}(B, R^p \check f_{\ast} \Z_2) \cong  H^{2p}(\check X, \Z_2)$.  The map $\beta$ in this case is the squaring with respect to the usual cup product in cohomology.

\section{Short exact sequences}  \label{seseq}
Our first goal is to generalize the short exact sequence \eqref{long_seq} to the case of twisted real structures. We will prove the following result
\begin{thm} \label{sheaf_seq}
There exist sheaves $\mathcal L^{1}_{\tau}$ and $\mathcal L^{2}_{\tau}$ over $B$ and a short exact sequence 
\begin{equation} \label{short_seq1}
            0 \longrightarrow \mathcal L^{1}_{\tau}  \longrightarrow \pi_{\tau_\ast} \Z_2 \longrightarrow \mathcal L^{2}_{\tau} \longrightarrow 0,
 \end{equation} 
such that $\mathcal L^{1}_{\tau}$ and $\mathcal L^{2}_{\tau}$ are related to the topology of $X$ by the following short exact sequences
 \begin{equation} \label{short_seq2}
            0 \longrightarrow \Z_2  \longrightarrow \mathcal L^1_{\tau} \longrightarrow  R^1 f_{\ast} \Z_2 \longrightarrow 0,
 \end{equation}
 \begin{equation} \label{short_seq3}
            0 \longrightarrow R^2 f_{\ast} \Z_2  \longrightarrow \mathcal L^2_{\tau} \longrightarrow \Z_2 \longrightarrow 0.
 \end{equation}
\end{thm} 

\subsection{Classification of Lagrangian sections} \label{section_class} The Lagrangian sections of $f: X \rightarrow B$ are classified, up to Hamiltonian equivalence, by the group $H^{1}(B, j_{\ast} \Lambda^*)$. This can be seen as follows. Take some covering $\{ U_i \}_{i \in I}$ of $B$ by open sets homeomorphic to $3$-balls.  Let us assume we are away from singularities, i.e. assume $U_i \cap \Delta = \emptyset$. Given a Lagrangian section $\tau$, consider $\tau_{|U_i}: U_i \rightarrow T^*U_i / \Lambda^*$.  Since $U_i$ is homeomorphic to a $3$-ball, we can find a Lagrangian lift $\tilde \tau_{U_i}: U_i \rightarrow T^*U_i$ of $\tau_{|U_i}$.  Then, since $\tau$ is a global section, we must have that on overlaps $U_i \cap U_k$ 
\[ \tilde \tau_{U_i} - \tilde \tau_{U_k} \in \Lambda^*. \]
This can be extended to the case when $U_i \cap \Delta \neq \emptyset$.  Therefore, to the section $\tau$ we can associate the Cech $1$-cocycle $\{ U_i \cap U_k,  \tilde \tau_{U_i} - \tilde \tau_{U_k} \}$ giving a class in $H^1(B, j_{\ast} \Lambda^*)$, which we continue to denote by $\tau$. It can be shown that any class can be represented by a Lagrangian section and that two Lagrangian sections represent the same class if and only if they are Hamiltonian isotopic. 

\subsection{Local description of $\Sigma_{\tau}$} \label{sigmatau}Given the above local description of a Lagrangian section $\tau$, it is easy to describe the involution $\iota_{\tau}$ locally. We will do this away from singularities, i.e. when $U_i \cap \Delta = \emptyset$. Indeed 
\[ \begin{split}
        \iota_{\tau}: T^*U_i/ \Lambda^* & \rightarrow T^*U_i/ \Lambda^* \\
                                    [\alpha] & \mapsto [ -(\alpha + \tilde \tau_{U_i}) ]
   \end{split} \]
where $\alpha$ is a $1$-form and $[ \cdot]$ denotes the class in the quotient by $\Lambda^*$. Then, locally, we have 
\[ \Sigma_{\tau |_{U_i}} = - \tfrac{\tilde \tau_{U_i}}{2} + \tfrac{1}{2} \Lambda^* \mod \Lambda^* \]
In particular the fibre $\pi_{\tau}^{-1}(b)$ has the structure of an affine space modelled on $\Lambda^* \otimes \Z_{2}$. 

\subsection{The sheaves $\mathcal L^{1}_{\tau}$ and $\mathcal L^2_{\tau}$} The sheaf $\mathcal L^{1}_{\tau}$ is easily described. Its stalks over $b \in B_0$ are the affine maps $\pi^{-1}(b) \rightarrow \Z_2$, which embed inside $\pi_{\tau \ast} \Z_2$. Moreover, on $B_0$, we have an obvious sequence
\[
            0 \longrightarrow \Z_2  \longrightarrow \mathcal L^1_{\tau} \longrightarrow  \Lambda \otimes \Z_2 \longrightarrow 0,
\]
where the left hand side is the inclusion of the constant maps, while the righthand side is given by taking the linear part of an affine map. If we push this sequence forward by $j: B_0 \rightarrow B$ and we use simplicity we get the sequence \eqref{short_seq2}.  By definition $\mathcal L^2_{\tau}$ is the quotient of the inclusion $\mathcal L^{1}_{\tau} \rightarrow \pi_{\tau \ast} \Z_2$. 

\subsection{Proof of Theorem \ref{sheaf_seq}} \label{seq:proof} Let $S$ be an affine space modeled on a $\Z_2$-vector space $V$ of dimension $3$. Let $L^1_S = \operatorname{Aff}(S)$ be the space of affine functions on $S$.  Let $L^2_S$ be the quotient between  $\operatorname{Maps}(S, \Z_2)$ and $L^1_S$. Given a subset $A \subset S$, let $\delta_{A}$ denote the function which is $1$ on $A$ and $0$ elsewhere and by $[\delta_A]$ the class of $\delta_A$ in $L^2_S$. An affine function on $S$ is of the type $\delta_{W}$ or $1 + \delta_{W}$ for some affine subspace $W \subseteq S$ of codimension $1$ or $0$. It is not hard prove that $L^2_S$ is generated by the elements of type $[ \delta_Z]$, where $Z \subset S$ is either a line or a point. Moreover, given two lines $Z_1$ and $Z_2$, then $[\delta_{Z_1}] = [\delta_{Z_2}]$ if and only if $Z_1$ and $Z_2$ are parallel. We then have an exact sequence
\begin{equation} \label{seq_aff}
   0 \rightarrow V \rightarrow L^{2}_{S} \rightarrow \Z_2 \rightarrow 0 
\end{equation}
Where a vector $v \in V$ is mapped to $[\delta_{Z_v}]$ where $Z_v$ is a line with direction $v$ if $v \neq 0$ or the empty set if $v=0$. It can be easily proved that
\[ [\delta_{Z_{v+w}}] = [\delta_{Z_v}] +  [\delta_{Z_w}] \]
so that the first map is linear. The quotient of $L^{2}_{S}$ by $V$ is $\Z_2$ and it is generated by the class of $[\delta_{q}]$ where $q$ is a point in $S$. 
Given $b \in B_0$, let $S = \pi_{\tau}^{-1}(b)$. As we said, $S$ is an affine space modeled on $\Lambda^* \otimes \Z_2$.  Using the isomorphism $(R^2 f_{\ast} \Z_2)_b \cong \Lambda^* \otimes \Z_2$ (see \eqref{R2} and \eqref{R2-1}), the sequence \eqref{seq_aff} becomes the sequence \eqref{short_seq3}.

\subsection{The two dimensional case}\label{dimension2} If $B$ is two dimensional (hence the affine base of a K3 surface), then we only have the sheaf $\mathcal L^1_{\tau}$ of affine maps inside $\pi_*\Z_2$ which satisfies \eqref{short_seq2} and
\begin{equation} \label{L1dim2}
         0  \longrightarrow \mathcal L^1_{\tau} \longrightarrow \pi_{\ast} \Z_2 \longrightarrow \Z_2 \longrightarrow 0. 
\end{equation} 
This follows from the fact that if $S$ is an affine space over $\Z_2$ of dimension $2$ then we have
\[ 0 \rightarrow \operatorname{Aff}(S) \rightarrow \operatorname{Maps}(S, \Z_2) \rightarrow \Z_2 \rightarrow 0 \]
where $\Z_2$ is generated by $[\delta_q]$, with $q \in S$.

\section{The connecting homomorphism} \label{cnnctng}

The sequence  \eqref{short_seq1} gives the long exact sequence in cohomology
\begin{equation}  \label{connecting2}
\begin{split}
          0  \longrightarrow & H^{0}(B, \mathcal L^{1}_{\tau})  \longrightarrow H^{0}(\Sigma_{\tau}, \Z^2)   \longrightarrow   H^{0}(B,  \mathcal L^{2}_{\tau}) \longrightarrow \\
          \ & H^{1}(B, \mathcal L^{1}_{\tau}) \longrightarrow H^{1}(\Sigma_{\tau}, \Z^2)   \longrightarrow   H^{1}(B,  \mathcal L^{2}_{\tau} ) \stackrel{\beta}{\longrightarrow}  \\
          \ & H^{2}(B,  \mathcal L^{1}_{\tau} )  \longrightarrow  H^{2}(\Sigma_{\tau}, \Z_2) \longrightarrow H^{2}(B,  \mathcal L^{2}_{\tau}) \longrightarrow \ldots 
\end{split}
\end{equation}
Combining this with the maps induced by the sequences \eqref{short_seq2} and \eqref{short_seq3} we obtain the diagram
\begin{equation} \label{beta}  \begin{CD}
 H^{1}(B, R^2f_{\ast} \Z_2 ) @>\beta'>>  H^{2}(B,  R^1f_{\ast} \Z_2 ) \\
@VVV @AAA \\
H^{1}(B,  \mathcal L^{2}_{\tau} )  @>\beta>>  H^{2}(B,  \mathcal L^{1}_{\tau} )
\end{CD}\end{equation}
where $\beta'$ is obtained by composition. 

Using the isomorphisms \eqref{Riso}, we interpret $\beta'$ as a map on the cohomology of the mirror
\[  \beta':  H^{1}(B, R^1 \check f_{\ast} \Z_2) \rightarrow H^{2}(B, R^2 \check f_{\ast} \Z_2).\]
As explained in \S \ref{section_class},  the twist $\tau$ is a class in $H^1(B, j_* \Lambda^*)$. Notice that 
\[ H^1(B, j_* \Lambda^*) = H^1(B, R^1 \check f_{\ast} \Z). \]
This allows us to interpret $\tau$ as the class of a line bundle $L_{\tau}$ on $\check X$ (indeed, it is conjectured that Lagrangian sections are mirror to line bundles).  The assumption that $\tau$ is not a square (see \S \ref{twist_real}) implies that $L_{\tau}$ is not zero after reduction modulo two. Therefore, our assumption is that the twist $\tau$ is such that the mirror line bundle $L_{\tau}$ satisfies
\[ 0 \neq L_{\tau} \in H^1(B, R^1 \check f_{\ast} \Z_2). \]

We can now state our main result

\begin{thm} \label{beta_MS}
The map $\beta': H^{1}(B, R^1 \check f_{\ast} \Z_2) \rightarrow H^{2}(B, R^2 \check f_{\ast} \Z_2)$, related to the connecting homomorphism in the long exact sequence \eqref{connecting2} via diagram \eqref{beta}, coincides with the map 
\[ \begin{split} 
            S_{\tau} : H^{1}(B, R^1 \check f_{\ast} \Z_2) & \longrightarrow H^{2}(B, R^2 \check f_{\ast} \Z_2) \\
                                                                                            D  & \longmapsto D^2 +  DL_{\tau}
   \end{split}
            \]
\end{thm}

\begin{proof} The proof is similar to the proof of Theorem  \ref{arguz_prince} in \cite{arg_princ:realCY}. Fix an open covering $\mathfrak U = \{ \mathcal U_i \}_{i \in I}$ which is Leray for all the sheaves and such that all triple intersections do not intersect the discriminant $\Delta$.  We will denote multiple intersections of open sets by
\[ U_{i_0, \ldots, i_k} = U_{i_0} \cap \ldots \cap U_{i_k}. \]
The cup product in Cech cohomology has the following description. Let $\alpha \in H^p(B, \mathcal F)$ and $\beta \in H^q(B, \mathcal G)$ then the cup product $\alpha \cup \beta \in H^{p+q}(B, \mathcal F \otimes G)$ is represented by the cochain
\[ (\alpha \cup \beta)_{i_0 \ldots i_{p+q}} = \sum_{r=0}^{p+q} \alpha_{i_r, \ldots, i_{r+p}} \otimes \beta_{i_{r+p}, \ldots, i_{r+p+q}} \]
where we have chosen cocycles representing $\alpha$ and $\beta$. The indices in this formula should be interpreted cyclically.  

In the following, when we take a local section of $\Lambda^* \otimes \Z_2$ and denote it by $\lambda$, we will mean that $\lambda \in \Lambda^*$ reduced mod $2$, i.e. $\lambda$ will be short for $\lambda \otimes 1$.  On the other hand we can also identify
\[ \Lambda^* \otimes \Z_2 = \tfrac{1}{2} \Lambda^* \ \mod \Lambda^* \]
and therefore we may identify $\lambda$ with the point $\tfrac{1}{2} \lambda$ in the fibre $f^{-1}(b)$. 
Take a class $D \in H^{1}(B, R^1 \check f_{\ast} \Z_2)$ represented by a $1$-cycle $\{ U_{ij}, D_{ij} \}$.  The twisting cycle $L_{\tau}  \in H^{1}(B, R^1 \check f_{\ast} \Z_2)$ is represented by $\{ U_{ij}, \tau_{ij} \}$ as described in \S \ref{section_class}.  The product $D^2 + DL_{\tau}$ is then represented by the $2$-cycle
\begin{equation*} 
	\begin{split}
           (D^2 + DL_{\tau})_{ijk} = D_{ij} \wedge & (D_{jk} + \tau_{jk}) + D_{jk} \wedge (D_{ki} + \tau_{ki})  \\
                                                                         & + D_{ki} \wedge (D_{ij} + \tau_{ij}).
         \end{split}
\end{equation*}
We now compare this cycle with $\beta'(D)$. First, we need to view $D$ as a class in $H^{1}(B, \mathcal L^{2}_{\tau})$.
 As described in \S \ref{seq:proof}, $\{ U_{ij}, D_{ij} \}$ is sent to the cycle $\bar D = \{ U_{ij}, [\delta_{Z_{D_{ij}}}] \}$, where $Z_{D_{ij}}$ is a line with direction $D_{ij}$ if $D_{ij} \neq 0$ or the empty set if $D_{ij} =0$.  Now we need a cochain $\Gamma \in C^1(\mathfrak U, \pi_{\tau *}\Z_2)$ representing $\bar D$. Using the local description of $\Sigma_{\tau}|_{U_i}$ given in \S \ref{sigmatau}, we can represent $Z_{D_{ij}}$ as the line through the point $-\tfrac{\tilde \tau_i}{2}$ and direction $D_{ij}$. Therefore we lift $[\delta_{Z_{D_{ij}}}]$ to the map 
\[ \Gamma_{ij} = \delta_{Z_{D_{ij}}} = \delta_{-\tfrac{\tilde \tau_i}{2}} + \delta_{-\tfrac{\tilde \tau_i + D_{ij}}{2}} \]
Then we consider $\partial \Gamma \in C^2(\mathfrak U, \pi_{\tau *}\Z_2)$. On $U_{ijk}$ we have
\begin{equation}
 \begin{split}
     (\partial \Gamma)_{ijk}  & = \delta_{Z_{D_{ij}}} + \delta_{Z_{D_{jk}}}  + \delta_{Z_{D_{ki}}} = \\ 
                                          & = \delta_{-\tfrac{\tilde \tau_i}{2}} + \delta_{-\tfrac{\tilde \tau_i + D_{ij}}{2}}  \\
                                          & + \delta_{-\tfrac{\tilde \tau_j}{2}} + \delta_{-\tfrac{\tilde \tau_j + D_{jk}}{2}} \\
                                          & + \delta_{-\tfrac{\tilde \tau_k}{2}} + \delta_{-\tfrac{\tilde \tau_k + D_{ki}}{2}}
   \end{split} 
\end{equation} 
The next steps consist first in describing $(\partial \Gamma)_{ijk}$ as an affine function $\alpha_{ijk}$ on $\pi_{\tau}^{-1}(b)$. Thus giving a cycle $\alpha \in H^2(B, \mathcal L^{1}_{\tau})$. Notice that $\beta (\bar D) = \alpha $. Then we need to take the linear part $\beta_{ijk} = \operatorname{Lin}(\alpha_{ijk})$ so that $\beta'(D) = \{ U_{ijk}, \beta_{ijk} \}$. Then we compare $\beta_{ijk}$ with $(D^2 + DL_{\tau})_{ijk}$.

Let us identify $\pi_{\tau}^{-1}(b)$ with the vector space $V = \Lambda^* \otimes \Z_2$ by declaring the point $-\tfrac{\tilde \tau_i}{2}$ to be the origin. Moreover, let us denote by $Z_0, Z_1, Z_2$ respectively the sets $Z_{D_{ij}}$, $Z_{D_{jk}}$ and $Z_{D_{ki}}$.  Let
\[ e_1 = D_{jk}, \quad e_2 = D_{ki}, \quad f_1 = \tau_{ij}, \quad f_2 = \tau_{ki}. \]
The cocycle condition implies $D_{ij} = e_1+e_2$ and $\tau_{jk} = f_1 + f_2$. It is then easy to see that 
\[ \begin{split}
     \alpha_{ijk} = (\partial \Gamma)_{ijk}  & = \delta_{Z_0} + \delta_{Z_1}  + \delta_{Z_2} = \\ 
                                          & = \delta_{0} + \delta_{e_1+ e_2}  \\
                                          & + \delta_{f_1} + \delta_{f_1+e_1} \\
                                          & + \delta_{f_2} + \delta_{f_2+e_2}
   \end{split} \]
On the other hand we have
\begin{equation} 
    \begin{split} 
              (D^2 + DL_{\tau})_{ijk} & = (e_1 + e_2) \wedge  (e_1+ f_1+f_2) + e_1 \wedge (e_2 + f_2) + \\
                                                   & \ \ \ \ \ \ \ \ \ \ + e_2 \wedge (e_1+e_2 + f_1) = \\
                                                   & = e_1 \wedge e_2 + e_1 \wedge f_1 + e_2 \wedge f_2 
    \end{split}
\end{equation}
We study four different cases. 

\medskip

{\it Case 1: $e_1$ and $e_2$ are linearly dependent.} If $e_1 = e_2 = 0$, then both $(\partial \Gamma)_{ijk}$ and $(D^2 + DL_{\tau})_{ijk}$ are zero, in particular they match. Otherwise we may assume w.l.o.g. that $e_1 = e_2 \neq 0$. Then $Z_0 = \emptyset$ and $Z_1$ and $Z_2$ either coincide or are parallel.  In the first case $e_1$ and $f_1+f_2$ are linearly dependent, therefore $(\partial \Gamma)_{ijk}$ and $(D^2 + DL_{\tau})_{ijk}$ both vanish. Otherwise if $Z_1$ and $Z_2$ are parallel and distinct, then $e_1$ and $f_1 + f_2$ are linearly independent. In particular 
\[  (D^2 + DL_{\tau})_{ijk} = e_1 \wedge (f_1 +  f_2)
 \]
is a non-zero two form. On the other hand $\alpha_{ijk} = (\partial \Gamma)_{ijk} = \delta_{W}$, where $W$ is the unique $2$-plane containing the two lines. In particular $\alpha_{ijk}$ is a non constant affine function. Let $e_3$ be a third vector so that $\{e_1, f_1 + f_2, e_3 \}$ forms a basis of $V$. Let $\{e_1^*, e_2^*, e_3^* \}$ form the dual basis of $V^* = \Lambda \otimes \Z_2$.  Taking the linear part of $\alpha_{ijk}$ we have
\[ \beta_{ijk} = \operatorname{Lin} (\alpha_{ijk}) = e_3^*. \]
Consider $\Omega = e_1 \wedge (f_1 + f_2) \wedge e_3$, which coincides with the global $3$-form on $B$.  Contracting $\Omega$ with $e_3^*$ gives precisely $e_1 \wedge (f_1 + f_2)$. Therefore, after applying the isomorphism  \eqref{Riso} we have
\[  \beta_{ijk} = (D^2 + DL_{\tau})_{ijk}. \] 

\medskip

We now assume $e_1$ and $e_2$ are linearly independent. Let $e_3$ be a third vector so that $\{e_1, e_2, e_3 \}$ form a basis of $V$ and let $\{e_1^*, e_2^*, e_3^*\}$ be the dual basis of $V^* = \Lambda \otimes \Z_2$. As above $\Omega = e_1 \wedge e_2 \wedge e_3$ coincides with the global three form on $B$.  We discuss the following three cases. 

\medskip

{ \it Case 2: $Z_0, Z_1, Z_2$ are coplanar and pass through the same point.} In this case $\alpha_{ijk} = (\partial \Gamma)_{ijk} = \delta_{W}$, where $W$ is the unique plane containing the three lines. We have 
\[ \beta_{ijk} = \operatorname{Lin}(\alpha_{ijk}) = e^*_3. \]
Let $q$ be the common point of the three lines. We must have $q=0$ or $q=e_1 + e_2$. Notice that $f_j = q + \epsilon_j e_j$, where $\epsilon_j = 1, 0$. Then 
\[ 
              (D^2 + DL_{\tau})_{ijk}  = e_1 \wedge e_2 + (e_1  + e_2) \wedge q = e_1 \wedge e_2.
\]
Since contracting $\Omega$ with $e_3^*$ gives $e_1 \wedge e_2$, we have $\beta_{ijk} = (D^2 + DL_{\tau})_{ijk}$. 
\medskip

{ \it Case 3: $Z_0, Z_1, Z_2$ are coplanar and intersect pairwise at three different points.} In this case, $\alpha_{ijk} =0$. Let $q = Z_0 \cap Z_1$. Then $f_1 = q + \epsilon_1 e_1$ and $f_2 = q + e_1 + \epsilon_2 e_2$.  Therefore 
\[ 
              (D^2 + DL_{\tau})_{ijk} = e_1 \wedge e_2 + e_1 \wedge q + e_2 \wedge (q+e_1) = (e_1+ e_2) \wedge q = 0.
  \]
So that $\beta_{ijk} = (D^2 + DL_{\tau})_{ijk}$.
\medskip

{ \it Case 4: $Z_1, Z_2$ are coplanar and $Z_0$ is disjoint from $Z_1$ and $Z_2$.}  Let $q = Z_1 \cap Z_2$. We must have that $q$ is linearly independent from $e_1$ and $e_2$, therefore we may assume that $e_3 = q$.  We have that  $\alpha_{ijk} = \delta_{W}$, where $W$ is the unique $2$-plane containing $Z_0$ and the points $e_1+e_3$ and $e_2+e_3$. It can be easily seen that
\[ \beta_{ijk} = \operatorname{Lin}(\alpha_{ijk}) = e_1^* + e_2^* + e_3^* \]
On the other hand, we have $f_j = e_3+\epsilon_j e_j$, which gives
\[ 
     (D^2 + DL_{\tau})_{ijk}  =  e_1 \wedge e_2 + e_1 \wedge e_3 + e_2 \wedge e_3. 
\]
This is precisely the two form obtained by contracting $\Omega$ with $e_1^* + e_2^* + e_3^*$.  Therefore $\beta_{ijk} = (D^2 + DL_{\tau})_{ijk}$ also in this case. 

\medskip

This concludes the proof. 
\end{proof}

\section{Connectedness and bounds on Betti numbers} \label{applications}
We will compute some consequences on the cohomology of $\Sigma_{\tau}$ with the assumption that the base $B$ is a homology $\Z_2$-sphere and that $H^1(X, \Z_2)= H^1(\check X, \Z_2)=0$. We will prove the following
\begin{thm} \label{applic}
With the above assumptions, if $\tau$ is a non-trivial twist, then $\Sigma_{\tau}$ is connected and its Betti numbers satisfy
\[ b_1(\Sigma_{\tau}, \Z_2) \leq \dim H^1(B, R^1f_*\Z_2) + \dim H^1(B, R^2f_*\Z_2) - 1. \]
where equality holds if and only if the connecting homomorphism $\beta$ in \eqref{connecting2} is zero.
If the integral cohomologies of $X$ and $\check X$ have no $\Z_2$ torsion, then
\[ b_q(\Sigma_{\tau}, \Z_2) \leq  h^{q, 3-q}(X) + h^{q,q}(X) - 1. \]
 \end{thm} 

Notice in particular that $\Sigma_{\tau}$ is never maximal, in fact we have 
\[ \sum b_j(\Sigma_{\tau}, \Z_2) \leq \sum b_j(X, \Z_2) - 4. \]
When this inequality is an equality, $\Sigma_{\tau}$ is called an $(M-2)$ real variety ($M$ stands for maximal). 

\subsection{Cohomology of $\mathcal L^2_\tau$} We prove the following 
\begin{lem} \label{cohL2} In the hypothesis of Theorem \ref{applic} we have
\[  
   \begin{split}
   & H^0(B, \mathcal L^2_\tau)  = 0 \\
   & H^1(B, \mathcal L^2_\tau) ) \cong \frac{H^1(B, R^2f_*\Z_2)}{\langle \tau \rangle} \\
   & H^2(B, \mathcal L^2_\tau)   \cong H^2(B, R^2f_*\Z_2) \\
   & H^3(B, \mathcal L^2_\tau) \cong \Z_2 
   \end{split}
 \]
\end{lem}
\begin{proof} 
It follows from the discussion in \S \ref{Leray} that
\[  
   \begin{split}
         & H^{0}(B, R^2f_*\Z_2) \cong H^{0}(B, R^1\check f_*\Z_2) = 0  \\
         & H^3(B, R^2f_*\Z_2) = 0 
     \end{split}
 \]
Moreover 
\[  H^{0}(B, \Z_2) =  H^{3}(B, \Z_2) = \Z_2, \quad H^{1}( B, \Z_2)  = 0 \]
since $B$ is a homology $\Z_2$ sphere. 
Hence the long exact sequence associated to  \eqref{short_seq3} splits as follows
\begin{equation}  \label{l2:coho}
{\small \begin{split}
         \ & 0  \longrightarrow H^{0}(B, \mathcal L^2_\tau)   \longrightarrow   \Z_2  \longrightarrow  H^{1}( B, R^2f_*\Z_2) \longrightarrow H^{1}(B, \mathcal L^2_\tau)   \longrightarrow  0 \\
    \ &    0  \longrightarrow  H^{2}( B, R^2f_*\Z_2) \longrightarrow H^{2}(B, \mathcal L^2_\tau) \longrightarrow 0 \\
    \ &    0  \longrightarrow H^{3}(B, \mathcal L^2_\tau)  \longrightarrow \Z_2 \longrightarrow 0 
\end{split} }
\end{equation}

 The last two lines give the last two statements of the lemma. In the first line we have two possibilities, either $H^0(B, \mathcal L^2_{\tau}) = 0$ or  $H^0(B, \mathcal L^2_{\tau}) \cong H^0(B, \Z_2) \cong \Z_2$. Let us prove that the former holds. 
 Take some covering $\mathfrak U=\{U_i \}$ over which the cycle $\tau$ and $\Sigma_{\tau}$ can be described as in \S \ref{section_class}. 
 Then, over each $U_i$ we have identifications
 \[ (\Lambda^* \otimes \Z_2) \oplus \Z_2 \stackrel{\phi_i}{\longrightarrow} \mathcal L^2|_{U_i}. \]
 In fact for every non-zero $v \in  (\Lambda^* \otimes \Z_2)$ let $Z_v$ be the line with direction $v$ and passing through $- \tfrac{\tilde{\tau}_i}{2}$. Then we define
 \[ \phi_i(v, \epsilon) = \left[  \delta_{Z_v} + \epsilon \delta_{- \tfrac{\tilde{\tau}_i}{2}} \right]. \]
 It is then easy to check that over $U_{ij}$
 \[ \phi_j^{-1} \circ \phi_i(v, \epsilon) = (v + \epsilon \tau_{ij}, \epsilon). \]
 Suppose by contradiction that $\mathcal L^2_\tau$ has a non-trivial section $\alpha$ which is mapped to $1$ under the homomorphism $H^0(B, \mathcal L^2_{\tau}) \rightarrow \Z_2$. Then, locally with the above identifications $\phi_i$, we have 
 \[ \alpha_i = \alpha|_{U_i} = (v_i, 1) \]
 for some local section $v_i$ of $\Lambda^* \otimes \Z_2$. But since $\alpha$ is a section, we must have that on $U_{ij}$ 
 \[ (v_i + \tau_{ij} , 1) = (v_j, 1). \]
This implies that $\tau_{ij} = v_j - v_i$, i.e. that $\tau$ is the trivial class, contradicting our assumption. 

The first line of \eqref{l2:coho} becomes
\[     0  \longrightarrow \Z_2 \longrightarrow  H^{1}( B, R^2f_*\Z_2) \longrightarrow H^{1}(B, \mathcal L^2_\tau)   \longrightarrow   0.
\]
We now prove that image of $\Z_2$ inside $H^{1}( B, R^2f_*\Z_2)$ is generated by $\tau$. Given $1 \in \Z_2 = H^0(B, \Z_2)$, we first lift it to a cochain  $\gamma \in C^0(\mathfrak U, \mathcal L^2_\tau)$. Then $\partial \gamma$ comes from a cocycle $\lambda$ in $C^1( \mathfrak U, R^2f_*\Z_2)$ whose class is the image of $1$. We can define $\gamma$ on each $U_i$ by
\[ \phi_i(\gamma_i )= (0, 1). \]
Now we have 
\[  \phi_j(\gamma_i - \gamma_j) = \phi_i(\gamma_i - \gamma_j) = (\tau_{ij}, 0). \]
Therefore the cocycle $\lambda$ coincides with $\tau$. This proves the second isomorphism in this lemma. 
\end{proof}

\subsection{Cohomology of $\mathcal L^1_\tau$}
We can do a similar analysis of the cohomology of $\mathcal L^1_\tau$. 
\begin{lem} \label{cohL1} In the hypothesis of Theorem \ref{applic} we have that
\[  
   \begin{split}
      & H^0(B, \mathcal L^1_\tau)  = \Z_2 \\
    & H^1(B, \mathcal L^1_\tau) \cong H^1(B, R^1f_*\Z_2) 
   \end{split}
 \]
 \end{lem}

\begin{proof}  Both isomorphism follow from the long exact sequence associated to \eqref{short_seq2}. 
\end{proof}

\subsection{Proof of Theorem  \ref{applic}} We apply the isomorphisms of Lemmas \ref{cohL2} and \ref{cohL1} to the long exact sequence \eqref{connecting2}. The first line becomes 
\begin{equation}  \label{connected}
                    0  \longrightarrow  \Z_2  \longrightarrow H^{0}(\Sigma_{\tau}, \Z_2)   \longrightarrow   0.
\end{equation}
This proves that $\Sigma_{\tau}$ is connected.  The rest of \eqref{connecting2}  becomes
\begin{equation}  \label{connecting3}
\begin{split}
          0  \longrightarrow   \ & H^1(B, R^1f_*\Z_2) \longrightarrow H^{1}(\Sigma_{\tau}, \Z_2)   \longrightarrow   H^{1}(B,  \mathcal L^{2}_{\tau} ) \stackrel{\beta}{\longrightarrow}  \\
          \ & H^{2}(B,  \mathcal L^{1}_{\tau} )  \longrightarrow  H^{2}(\Sigma_{\tau}, \Z_2) \longrightarrow H^{2}(B,  \mathcal L^{2}_{\tau}) \longrightarrow \ldots,
\end{split}
\end{equation}
which gives
\[ 
\begin{split}
    b_1(\Sigma_{\tau}, \Z_2) & \leq \dim H^1(B, R^1f_*\Z_2) +  \dim H^{1}(B,  \mathcal L^{2}_{\tau} ) \\
                                            & = \dim H^1(B, R^1f_*\Z_2) + \dim H^1(B, R^2f_*\Z_2) - 1.
\end{split}
\]
Obviously the equality holds if and only if $\beta = 0$. If the integral cohomology of X has no $\Z_2$ torsion then the dimensions of the spaces on the righthand side equal the corresponding Hodge numbers.

\subsection{Topology of twisted real $K3$ surfaces} Assume that $B$ has dimension $2$ and that it is the affine base of a $K3$ surface.  Let $\Sigma_{\tau}$ be the real twisted $K3$ associated to some twist $\tau \in H^{1}(B, R^1 \check f_{\ast} \Z_2)$.

\begin{thm} If the twist is non trivial, then the real twisted $K3$ surface $\Sigma_{\tau}$ is connected and has genus $9$. 
\end{thm}

\begin{proof} We use the sheaf $\mathcal L^1_{\tau}$ with the properties described in \S \ref{dimension2}.  The long exact sequence associated to \eqref{L1dim2} splits as
\[
\begin{split}
         0  \longrightarrow & H^{0}(B, \mathcal L^1_\tau)   \longrightarrow   H^0(\Sigma_\tau,  \Z_2) \longrightarrow  \Z_2 \stackrel{\alpha}{\longrightarrow}  \\
          \ & \longrightarrow H^{1}(B, \mathcal L^1_\tau)  \longrightarrow    H^1(\Sigma_\tau,  \Z_2)  \longrightarrow  0,
\end{split}
\]
and 
 \[ 0  \longrightarrow H^{2}(B, \mathcal L^1_\tau)   \longrightarrow   H^2(\Sigma_\tau,  \Z_2) \longrightarrow  \Z_2 \longrightarrow 0. \]
 
 The sequence \eqref{short_seq2} gives that $H^{0}(B, \mathcal L^1_{\tau}) \cong \Z_2$ and 
 \[ 0  \longrightarrow H^{1}(B, \mathcal L^1_\tau)   \longrightarrow   H^1(B,  R^1f_{\ast} \Z_2) \longrightarrow  \Z_2 \longrightarrow  H^{2}(B, \mathcal L^1_\tau) \longrightarrow 0. \]
 Let us prove that the homomorphism $\alpha$ in the first sequence is injective if and only if $\tau$ is non-trivial. The argument is similar to the proof of Lemma \ref{cohL2}. Take some covering $\mathfrak U=\{U_i \}$ over which the cycle $\tau$ and $\Sigma_{\tau}$ can be described as in \S \ref{section_class}. Then, over each $U_i$ we have identifications
 \[ \mathcal L^1_{\tau}|_{U_i} \stackrel{\phi_i}{\longrightarrow} (\Lambda \otimes \Z_2) \oplus \Z_2  . \]
 which map an affine function $\beta$ on $\Sigma_{\tau}|_{U_i}$ to $(v_{\beta}, \epsilon_{\beta})$, where $v_{\beta}$ is the linear part of $\beta$ and $\epsilon_{\beta} = \beta(- \tilde \tau_{U_i} / 2)$. It is easy to show that 
 \[  \phi_j \circ \phi_i^{-1}(v, \epsilon) = (v, \epsilon + v(\tau_{ij}) )\] 

Using the mirror symmetry isomorphism  \eqref{Riso}, the twist $\tau$ has a mirror $\check \tau \in  H^{1}(B, R^1 \check f_{\ast} \Z_2)$, which is represented by $\{ \check \tau_{ij} \}$, where $\check \tau_{ij} \in (\Lambda \otimes \Z_2)_{U_{ij}}$ is obtained by contracting $\tau_{ij}$ with a global integral $2$-form.  

 The generator $1$ of $\Z_2$ (i.e. the last term in the sequence \eqref{L1dim2}) can be represented on each $U_i$ as $[ \delta_{- \frac{\tilde \tau_{U_i}}{2}}]$. Then one can check that $\alpha(1)$ is represented by the $1$-cycle $\{ (\check \tau_{ij}, 1) \}$. This cycle is non-zero if and only if $\tau$ is non zero. Now the statement follows from the above sequences, where we use the fact that $\dim H^1(B,  R^1f_{\ast} \Z_2) = 20$.
\end{proof}

\section{The Renaudineau-Shaw spectral sequence} \label{rs:sq}
We will compare our results with the work \cite{ren_shaw:bound} by Renaudineau and Shaw, where they construct a spectral sequence computing the Betti numbers of real hypersurfaces in toric varieties arising from primitive patchworking.  In their case they consider the $n$-dimensional tropical hypersurface $X$, corresponding to the complex hypersurface $\C X$, and they compute the homology with $\Z_2$ coefficients of the real hypersurface $\R X$. 

\subsection{The results of Renaudineau-Shaw.} The patchworking data (i.e. the choice of signs) are encoded in a cosheaf $\mathcal{S}$ of $\Z_2$ modules over $X$, called the sign cosheaf. Its homology satisfies $H_q(X, \mathcal{S}) = H_q(\R X, \Z_2) $.  Then they construct a filtration of $\mathcal S$
\[ 0 = \mathcal K_{n+1} \subset \mathcal K_n \subset \ldots \subset \mathcal K_0 = \mathcal S\]
by cosheaves whose quotients
\[ \mathcal F_p = \frac{ \mathcal K_p}{\mathcal K_{p+1}} \]
are a $\Z_2$ version of the cosheaves defined by Itenberg, Katzarkov, Mikhalkin and Zharkov (IKMZ) in \cite{trop_hom}. The tropical homology groups of $X$ are defined as $H_q(X, \mathcal F_p)$. The above filtration induces a spectral sequence, converging to the homology of $\R X$, whose first page is given by $E^1_{q,p} = H_q(X, \mathcal F_p)$. The boundary morphisms are maps $\partial: H_{q+1}(X, \mathcal F_{p}) \rightarrow H_{q}(X, \mathcal F_{p+1})$.  In the original definition of \cite{trop_hom}, the cosheaves $\mathcal F_p$ are defined over $\Q$ and in this case IKMZ prove that 
\[ \dim(H_q(X, \mathcal F_p)) = h^{p,q}(\C X). \]
Under the assumption that the integral tropical homology of $X$ has no torsion, the spectral sequence implies that 
\begin{equation} \label{bounds:bh}
     b_q(\R X, \Z_2) \leq \begin{cases} h^{q,q}(\C X)  \ \text{if} \ q = n/2, \\ 
                                                          h^{q, n-q}(\C X) + h^{q,q}(\C X) \ \text{otherwise}. 
                                                          \end{cases}.
\end{equation}
It follows from \cite{a_r_s:lefsch} that indeed these inequalities hold when the ambient toric variety is smooth. This proves a generalization of the Itenberg conjecture \cite{itenberg:trophom_betti}. 
Renaudineau and Shaw also prove that $\R X$ is maximal if and only if the spectral sequence degenerates at the first page, i.e. if and only if all the boundary maps vanish. They conjecture that the spectral sequence degenerates at the second page.

\subsection{Comparison with our case.}  In our case, $B$ plays the role of the tropical variety $X$ and the sheaf $\pi_{\tau*} \Z_2$ is the replacement for the sign cosheaf $\mathcal S$. Notice that we use a sheaf because we compute cohomology, instead of homology. Let us define a filtration of $\pi_{\tau*} \Z_2$ analogous to the one described above. Let $V$ be an $n$-dimensional $\Z_2$-vector space and let $S_V = \operatorname{Maps}(V, \Z_2)$. We define the subspace $K^p \subset S_V$ as the set of maps which can be defined by a polynomial on $V$ of degree less or equal to $p$.  Then $K^0 \cong \Z^2$ is given by the constant maps, $K^1$ is the space of affine maps and it can be shown that $K^{n} = S_V$. Therefore we have a filtration 
\[ 0 \subset K^0 \subset K^1\subset .... \subset K^n = S_V \]
It can be shown (e.g. compare with Section 4 of \cite{ren_shaw:bound}) that 
\[ K^p / K^{p-1} \cong \bigwedge^p V^* \]
If $V$ is an affine space instead of a vector-space, choosing a point $V$ as the origin turns it into a vector space and thus we can still define the filtration on $\operatorname{Maps}(V, \Z_2)$. It can be shown that this filtration is independent of the chosen point. 

We have seen that at a smooth point $b \in B_0$,  $\pi_{\tau}^{-1}(b)$ is an affine space modeled on $\Lambda^*_b \otimes \Z_2$ and therefore we can define a filtration by sheaves on $\pi_{\tau \ast} \Z_2$, since $ (\pi_{\tau \ast} \Z_2)_b =  \operatorname{Maps}(\pi_{\tau}^{-1}(b), \Z_2)$:
\[ 0 \subset \mathcal K^1 \subset \mathcal K^2 \subset \ldots \subset \mathcal K^n = \pi_{\tau \ast}\Z_2. \]
Let us denote
\[ \mathcal F^p =  \mathcal K^p / \mathcal K^{p-1}. \]
Then we have 
\[ \mathcal F^p \cong \bigwedge^p \Lambda = R^p f_{\ast} \Z_2, \]
where the last equality follows from \eqref{Rlam}. Therefore the sheaves $R^p f_{\ast} \Z_2$ play the same role as the cosheaves $\mathcal F_p$ in the Renaudineau-Shaw sequence. The above filtration induces a spectral sequence computing the cohomology of $\Sigma_{\tau}$. The first page is $E^{1}_{q,p} = H^q( B, R^p f_{\ast} \Z_2)$ and the boundary maps are 
\[ \partial_{\tau}: H^q( B, R^{p+1} f_{\ast} \Z_2) \rightarrow H^{q+1}( B, R^p f_{\ast} \Z_2) \]

Let us now compare this spectral sequence with our sequences. We restrict to the case $n=3$, so that $\mathcal K^3 = \pi_{\tau*} \Z_2$. Notice that by definition 
\[ \mathcal L^1_{\tau} = \mathcal K^1. \] 
Moreover  the sequence \eqref{short_seq1} defining $\mathcal L^2_{\tau}$ implies that 
\[ \mathcal L^2_{\tau} = \mathcal K_3 / \mathcal K_1. \]
The sequence  \eqref{short_seq2} is the sequence 
\[ 0 \rightarrow \mathcal K^0 \rightarrow \mathcal K^1 \rightarrow \mathcal F^1 \rightarrow 0 \]
defining $\mathcal F^1$.  The sequence \eqref{short_seq3} corresponds to 
\begin{equation} \label{l2k}
  0 \rightarrow \mathcal F^2 \rightarrow \mathcal K^3 / \mathcal K^1 \rightarrow \mathcal F^3 \rightarrow 0.
\end{equation}

It is not difficult to show that the homomorphism $\beta'$ defined in \eqref{beta} coincides with the corresponding boundary map from the spectral sequence. Moreover the connecting homomorphism 
\[ H^0(B, \mathcal F^3) \rightarrow H^{1}(B, \mathcal F^2) \]
from sequence \eqref{l2k} is also a boundary map from the spectral sequence and it is equivalent to the connecting homomorphism 
\[ \Z_2 \rightarrow H^{1}(B, R^2 f_{\ast} \Z_2) \]
from \eqref{short_seq3}. In particular Lemma \ref{cohL2} shows that, with the given hypothesis, this map is always injective when the twist $\tau$ is non trivial. Therefore, in the hypothesis of \S \ref{applications}, the spectral sequence never degenerates at the first page when $\tau$ is non trivial. 

\subsection{Mirror symmetry} Let us now consider mirror symmetry. Applying the isomorphism \eqref{Riso}, the boundary homomorphisms for the first page of the spectral sequence become 
\[ \partial_{\tau}: H^q( B, R^{n-p-1} \check f_{\ast} \Z_2) \rightarrow H^{q+1}( B, R^{n-p} \check f_{\ast} \Z_2).\]
When $n=3$, Theorem \ref{beta_MS} and Lemma \ref{cohL2} explicitly describe these homomorphisms in the cases $(q,p)=(1,1)$ and $(q,p)=(0,2)$.   

\medskip

\noindent {\bf Question.} {\it Can we explicitly describe these homomorphisms for all $n, p$ and $q$?}

\medskip 

For instance, in three dimensions, there is one more case to compute in order to determine the spectral sequence, i.e. $(q,p)=(2,0)$. In fact, if the boundary map in this case is injective, then the spectral sequence degenerates at the second page. We were not able to find an explicit description of this case.

\section{A connected $(M-2)$-real quintic} \label{m-2:quintic}

In the case of the quintic threefold, with the torus fibration described by Gross in \cite{TMS},  Arguz-Prince \cite{arg_princ:realCY} computed the cohomology of the untwisted real quintic (i.e. $L_{\tau} =0$). They found that the map $\beta$ has rank $73$, and hence that $b_1(\Sigma) = 29$. In particular the untwisted real quintic is not maximal. Since also the twisted quintics are not maximal, as proved in Section \ref{applications}, none of the real quintics constructed in this way is maximal. The highest possible value of $b_1$ for our twisted real Calabi-Yau's is when $\beta_\tau = 0$. In this case, as proved in Section \ref{applications}, the real Calabi-Yau is an $(M-2)$-real variety. We prove the following
\begin{thm} \label{quasi_max} There exists a connected $(M-2)$ real twisted quintic $\Sigma_{\tau}$. In particular
\[ b_1(\Sigma_{\tau}, \Z_2) = 101. \]
\end{thm}

We will prove this by finding a divisor $L$ in the mirror quintic $\check X$ such that
\begin{equation} \label{condition_L}
            D^2 + DL = 0 \quad \forall D \in H^2(\check X, \Z_2).
\end{equation} 

\subsection{The mirror quintic} Consider the same simplex $P$ as in \S \ref{aff_quint}, i.e. the one with vertices 
\[ \begin{split}
        V_0=(-1, -1, & -1, -1),  \quad V_1= (4, -1, -1, -1), \quad  V_2= (-1, 4, -1, -1),  \\ 
        & V_3= ( -1, -1, 4, -1),  \quad V_4= (-1, -1, -1, 4 ). 
    \end{split} 
\]
It is a reflexive polytope with the origin as its only interior integral point. The fan whose cones are the cones over the faces of $P$ gives a singular toric variety $\check Y_{P}$. We can resolve $\check Y_{P}$ by taking a unimodular regular subdivision of the boundary $\partial P$ and take the associated fan, which gives a smooth toric variety $\check Y$. The mirror quintic $\check X$ is a smooth anti-canonical divisor in $\check Y$.  Let us take, as in Gross \cite{TMS}, a subdivision of $\partial P$ which on $2$-dimensional faces looks like Figure \ref{triang_2}.

\begin{figure}[!ht] 
\begin{center}
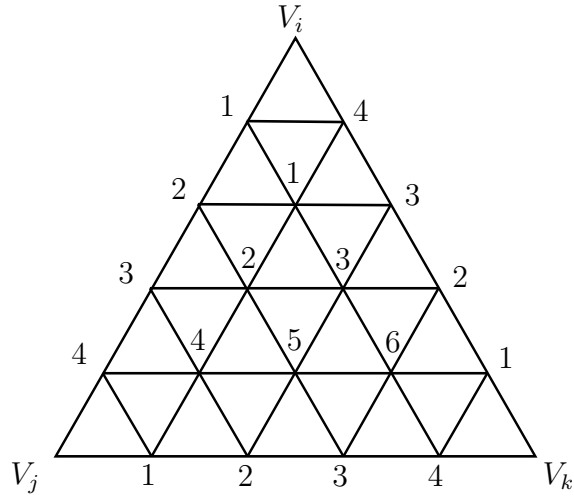
\caption{Triangulation of $2$-dimensional faces}
 \label{triang_2}
\end{center}
\end{figure}

\subsection{Divisors in the mirror quintic}
Each vertex in the subdivision of $\partial P$ corresponds to a one dimensional cone of the fan, hence to a toric divisor in $\check Y$. The divisors corresponding to vertices inside two dimensional faces of $\partial P$ are precisely the ones which have non trivial intersection with the mirror quintic $\check X$, therefore they correspond to non-zero divisors in $\check X$. Gross shows that these divisors generate $H^2(\check X, \Z)$, and hence also $H^2(\check X, \Z_2)$ (see Lemma 4.3 of \cite{TMS}). With some abuse of notation, we denote by $V_0, \ldots, V_4$ also the divisors corresponding to the vertices of $\partial P$. Moreover we denote by $E^{\ell}_{ij}$ the divisor corresponding to the $\ell$'th interior vertex along the edge from $V_i$ to $V_j$, where interior vertices of edges are numbered as in Figure \ref{triang_2}. The divisors in the interior of the $2$-face with vertices $V_i$, $V_j$ and $V_k$, numbered as in Figure \ref{triang_2}, are denoted by $F^{\ell}_{ijk}$.  In Proposition 4.2 of op. cit. Gross also computes triple intersection numbers $D_1 D_2 D_3$ between these divisors. We report here the $\mod 2$ version.  

\begin{prop} [Proposition 4.2 of \cite{TMS})]  \label{prop:trip_inters} The $\mod 2$ triple intersection numbers between the above divisors in $\check X$ are as follows
\begin{enumerate} 
\item $(V_i)^3 = (E^\ell_{ij})^3 = 1$ and $(F^{\ell}_{ijk})^3 = 0$.
 \item Given two distinct divisors $D_1$ and $D_2$ lying on the same $2$-face, then $(D_1)^2 D_2 = 1$ if and only if $D_1$ and $D_2$ are connected by one edge in the graph depicted in Figure \ref{trip_inters}. 
 \item Given three distinct divisors $D_1$, $D_2$, $D_3$ lying on the same $2$-face, then $D_1 D_2 D_3 = 1$ if and only if they are vertices of a two simplex in Figure \ref{triang_2}. 
\end{enumerate}
All other triple intersections are zero. 
\end{prop}

\begin{figure}[!ht]
\begin{center}
\includegraphics{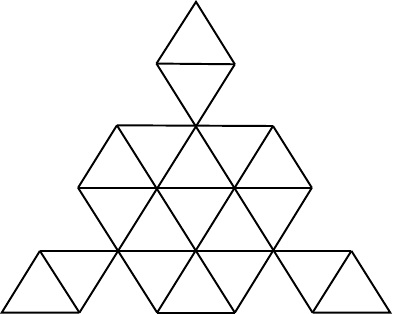}
\caption{Triple intersection graph}
\label{trip_inters}
\end{center}
\end{figure}

Let $S$ be the set of all divisors of type $V_i$, $E^{\ell}_{ij}$ or $F^{\ell}_{ijk}$. A general divisor in $H^2(\check X, \Z_2)$ can be written as
\[ L = \sum_{D \in S} \epsilon_D D \]
where $\epsilon_D \in \Z_2$. We can obviously view $L$ as a subset of $S$, where $D \in L$ if and only if $\epsilon_D =1$. We have that $L$ satisfies \eqref{condition_L} if and only if 
\[ D_1^2 D_2 = L D_1 D_2, \quad \forall D_1, D_2 \in S. \]

\subsection{Local configurations} It follows from Proposition \ref{prop:trip_inters} that if $D_1$ and $D_2$ do not belong to the same $2$-face then $D_1^2 D_2 = L D_1 D_2 = 0$. We now consider the following cases (and subcases)
\begin{enumerate}
\item $D_1$ and $D_2$ lie in the same $2$-face but not in the same edge of $\partial P$;
 \begin{itemize}
   \item[(1.1)] $D_1 \neq D_2$ and they are connected by an edge of the subdivision;
   \item[(1.2)] $D_1 \neq D_2$ and they are not connected by an edge of the subdivision;
   \item[(1.3)] $D_1 = D_2$ and it is in the interior of a $2$-face. 
 \end{itemize}
\item $D_1$ and $D_2$ lie inside the same edge. 
   \begin{itemize} 
           \item[(2.1)] $D_1 = D_2$;
           \item[(2.2)] $D_1 \neq D_2$ and they are connected by an edge in the graph of Figure \ref{trip_inters};
           \item[(2.3)] $D_1 \neq D_2$ and they are not connected by an edge in the graph of Figure \ref{trip_inters};
    \end{itemize}
\end{enumerate}

We now see how in the above cases the condition $D_1^2 D_2 = L D_1 D_2$ imposes certain local configurations on $L$.  

\medskip

\begin{figure}[!ht]
\begin{center}
\begingroup%
  \makeatletter%
  \providecommand\color[2][]{%
    \errmessage{(Inkscape) Color is used for the text in Inkscape, but the package 'color.sty' is not loaded}%
    \renewcommand\color[2][]{}%
  }%
  \providecommand\transparent[1]{%
    \errmessage{(Inkscape) Transparency is used (non-zero) for the text in Inkscape, but the package 'transparent.sty' is not loaded}%
    \renewcommand\transparent[1]{}%
  }%
  \providecommand\rotatebox[2]{#2}%
  \newcommand*\fsize{\dimexpr\f@size pt\relax}%
  \newcommand*\lineheight[1]{\fontsize{\fsize}{#1\fsize}\selectfont}%
  \ifx\svgwidth\undefined%
    \setlength{\unitlength}{97.29152649bp}%
    \ifx\svgscale\undefined%
      \relax%
    \else%
      \setlength{\unitlength}{\unitlength * \real{\svgscale}}%
    \fi%
  \else%
    \setlength{\unitlength}{\svgwidth}%
  \fi%
  \global\let\svgwidth\undefined%
  \global\let\svgscale\undefined%
  \makeatother%
  \begin{picture}(1,0.85915659)%
    \lineheight{1}%
    \setlength\tabcolsep{0pt}%
    \put(0,0){\includegraphics[width=\unitlength,page=1]{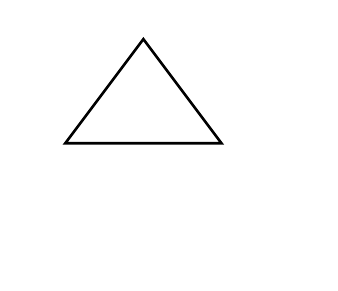}}%
    \put(-0.00777904,0.40778072){\color[rgb]{0,0,0}\makebox(0,0)[lt]{\lineheight{1.25}\smash{\begin{tabular}[t]{l}$D_1$\end{tabular}}}}%
    \put(0.66615316,0.40778072){\color[rgb]{0,0,0}\makebox(0,0)[lt]{\lineheight{1.25}\smash{\begin{tabular}[t]{l}$D_2$\end{tabular}}}}%
    \put(0.34251802,0.78006118){\color[rgb]{0,0,0}\makebox(0,0)[lt]{\lineheight{1.25}\smash{\begin{tabular}[t]{l}$D_3$\end{tabular}}}}%
    \put(0.35372623,0.01811765){\color[rgb]{0,0,0}\makebox(0,0)[lt]{\lineheight{1.25}\smash{\begin{tabular}[t]{l}$D_4$\end{tabular}}}}%
    \put(0,0){\includegraphics[width=\unitlength,page=2]{local_conf_1.pdf}}%
  \end{picture}%
\endgroup%

\caption{If $D_1$ and $D_2$ are on the same $2$-face and are connected by an edge, then only an odd number of the $D_1, D_2, D_3, D_4$ can be in $L$.}
 \label{local_conf_1}
\end{center}
\end{figure}

{\bf Case 1.1.} Let $D_3$ and $D_4$ be the other vertices of the two simplices which contain the edge from $D_1$ to $D_2$ as in Figure \ref{local_conf_1}. Then, using Proposition \ref{trip_inters}, we have
\[ D_1^2 D_2 = 1 \quad \text{and} \quad LD_1D_2=  \sum_{j=1}^{4} \epsilon_{D_j}. \] 
Therefore we have $D_1 ^2 D_2 = LD_1 D_2$ if and only if only an odd number of the vertices $D_1, D_2, D_3, D_4$ belong to $L$. 

\medskip 

{\bf Case 1.2.}  It is easy to see that in this case, both $D_1^2D_2$ and $LD_1D_2$ are always zero. 

\medskip

{\bf Case 1.3.} If $D_1 = D_2$ and it is in the interior of a $2$-faces of $\partial P$, then $D_1^2 D_2 = D_1^3 = 0$. On the other hand
\[ LD_1^2 =  \sum_{j=3}^{8} \epsilon_{D_j}, \]
where $D_3, \ldots, D_8$ are the six vertices adjacent to $D_1$. 

Therefore in this case, $D_1 ^2 D_2 = LD_1 D_2$ if and only if an even number of the $D_3, \ldots, D_8$ belong to $L$ (see Figure \ref{local_conf_2}). 

\begin{figure}[!ht] 
\begin{center}
\begingroup%
  \makeatletter%
  \providecommand\color[2][]{%
    \errmessage{(Inkscape) Color is used for the text in Inkscape, but the package 'color.sty' is not loaded}%
    \renewcommand\color[2][]{}%
  }%
  \providecommand\transparent[1]{%
    \errmessage{(Inkscape) Transparency is used (non-zero) for the text in Inkscape, but the package 'transparent.sty' is not loaded}%
    \renewcommand\transparent[1]{}%
  }%
  \providecommand\rotatebox[2]{#2}%
  \newcommand*\fsize{\dimexpr\f@size pt\relax}%
  \newcommand*\lineheight[1]{\fontsize{\fsize}{#1\fsize}\selectfont}%
  \ifx\svgwidth\undefined%
    \setlength{\unitlength}{124.15259053bp}%
    \ifx\svgscale\undefined%
      \relax%
    \else%
      \setlength{\unitlength}{\unitlength * \real{\svgscale}}%
    \fi%
  \else%
    \setlength{\unitlength}{\svgwidth}%
  \fi%
  \global\let\svgwidth\undefined%
  \global\let\svgscale\undefined%
  \makeatother%
  \begin{picture}(1,0.67420445)%
    \lineheight{1}%
    \setlength\tabcolsep{0pt}%
    \put(0,0){\includegraphics[width=\unitlength,page=1]{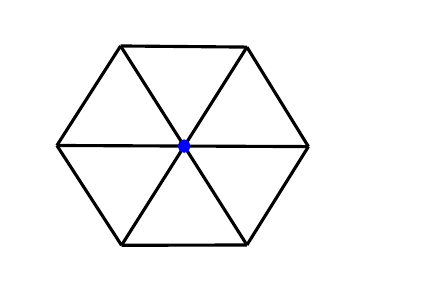}}%
    \put(0.36919448,0.42610204){\color[rgb]{0,0,0}\makebox(0,0)[lt]{\lineheight{1.25}\smash{\begin{tabular}[t]{l}$D_1$\end{tabular}}}}%
    \put(0.57362101,0.61222174){\color[rgb]{0,0,0}\makebox(0,0)[lt]{\lineheight{1.25}\smash{\begin{tabular}[t]{l}$D_3$\end{tabular}}}}%
    \put(0.23494416,0.61222174){\color[rgb]{0,0,0}\makebox(0,0)[lt]{\lineheight{1.25}\smash{\begin{tabular}[t]{l}$D_4$\end{tabular}}}}%
    \put(-0.00609602,0.31626091){\color[rgb]{0,0,0}\makebox(0,0)[lt]{\lineheight{1.25}\smash{\begin{tabular}[t]{l}$D_5$\end{tabular}}}}%
    \put(0.18612587,0.01724898){\color[rgb]{0,0,0}\makebox(0,0)[lt]{\lineheight{1.25}\smash{\begin{tabular}[t]{l}$D_6$\end{tabular}}}}%
    \put(0.54921191,0.0141978){\color[rgb]{0,0,0}\makebox(0,0)[lt]{\lineheight{1.25}\smash{\begin{tabular}[t]{l}$D_7$\end{tabular}}}}%
    \put(0.73838271,0.31320973){\color[rgb]{0,0,0}\makebox(0,0)[lt]{\lineheight{1.25}\smash{\begin{tabular}[t]{l}$D_8$\end{tabular}}}}%
  \end{picture}%
\endgroup%

\caption{If $D_1=D_2$ is in the interior of a $2$-face then an even number of the $D_3, \ldots, D_8$ can be in $L$.}
\label{local_conf_2}
\end{center}
\end{figure}
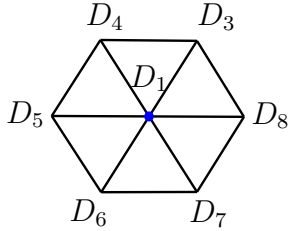
\medskip

{\bf Case 2.1.} If $D_1 = D_2$ and it lies on an edge of $\partial P$, then $D_1^2D_2 = D_1^3 = 1$. Let $S_{D_1}$ be the subset of $S$ consisting of $D_1$ and of all the vertices which are connected to $D_1$ via an edge of the graph in Figure \ref{trip_inters} for some $2$-face containing $D_1$. For example, if $D_1$ is of type $V_j$,  i.e. it is a vertex of $\partial P$, then $S_{D_1}$ contains $5$ elements, otherwise if $D_1$ is in the interior of an edge, $S_{D_1}$ will contain $8$ elements. Then we have
\[ LD_1D_2 = L {D_1^2} = \sum_{D \in S_{D_1}} \epsilon_{D}. \]

Therefore in this case $D_1 ^2 D_2 = LD_1 D_2$ if and only if $S_{D_1} \cap L$ contains an odd number of elements. 

\medskip

{\bf Case 2.2.}  Let $D_1 \neq D_2$ belong to some edge of $\partial P$ and assume they are connected by an edge in the graph of Figure \ref{trip_inters} (e.g. $D_1 = E^2_{ij}$ and $D_2 = E^3_{ij}$). Let $S_{D_1D_2}$ be the subset of $S$ consisting of $D_1$, $D_2$ and the vertices $D$ such that $D_1, D_2$ and $D$ are vertices of a $2$-simplex of the subdivision. In particular $S_{D_1D_2}$ contains $5$ elements. Then we have $D_1^2 D_2 = 1$ and 
\[ LD_1 D_2 = \sum_{D \in S_{D_1D_2}} \epsilon_{D}. \]
Therefore in this case $D_1 ^2 D_2 = LD_1 D_2$ if and only if $S_{D_1D_2} \cap L$ contains an odd number of elements. 
\medskip

{\bf Case 2.3.} Let $D_1 \neq D_2$ belong to some edge of $\partial P$ and assume they are not connected by an edge in the graph of Figure \ref{trip_inters}. We have two possibilities: either $D_1$ and $D_2$ are not adjacent (e.g. $D_1 = E^2_{ij}$ and $D_2 = E^4_{ij}$) or they are adjacent but the edge between them is not part of the graph in Figure \ref{trip_inters} (e.g. e.g. $D_1 = E^1_{ij}$ and $D_2 = E^2_{ij}$). In both cases we have $D_1^2D_2 = 0$. Let $S_{D_1D_2}$ be the subset of $S$ consisting of the vertices $D$ such that $D_1, D_2$ and $D$ are vertices of a $2$-simplex of the subdivision. Obviously $S_{D_1D_2}$ is empty in the first case and consists of $3$ elements in the second case. Then 
\[ LD_1 D_2 = \sum_{D \in S_{D_1D_2}} \epsilon_{D}. \] 
Therefore in this case $D_1 ^2 D_2 = LD_1 D_2$ if and only if $S_{D_1D_2} \cap L$ contains an even number of elements. 

\subsection{Proof of Theorem \ref{quasi_max} }
In Figure \ref{face_config} we give two examples of a $2$-face and divisor $L$ (the dotted ``red'' vertices) such that all interior edges satisfy the configurations of Case 1.1 and all interior vertices satisfy the configurations of Case 1.3. 
\begin{figure}[!ht] 
\begin{center}
\includegraphics{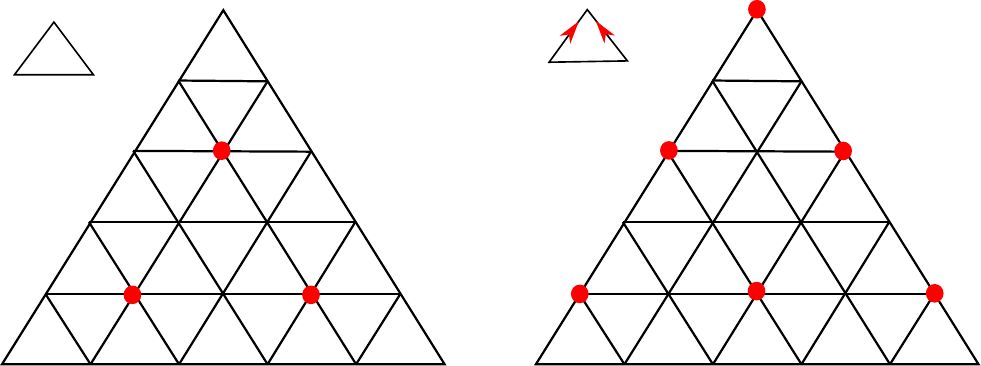}
\caption{The ``empty'' face and the ``arrow''}
\label{face_config}
\end{center}
\end{figure}
We call the picture on the left an ``empty face'' (because its edges are empty) and the one on the right  the ``arrow'' and we also depict their symbols. The arrows, in the symbol for the arrow, correspond to non-empty edges and they point towards the vertex which lies in $L$. 

In Figure \ref{edge_config} we give two examples of a neighborhood of an edge of $\partial P$ and a divisor $L$ such that every vertex of the edge satisfies the configurations of Case 2.1 and every edge satisfies the configurations of either Case 2.2 or Case 2.3. Since each edge of $\partial P$ is contained in three $2$-faces, we have glued together three copies (blue, black and red) of the graph in Figure \ref{trip_inters} along an edge. 
Notice that the configuration on the left can be obtained by gluing three copies of an ``arrow'' along a non-empty edge. The configuration on right can be obtained by gluing two copies of an ``empty'' face and one copy of an ``arrow'', along their empty edges. 
\begin{figure}[!ht] 
\begin{center}
\includegraphics{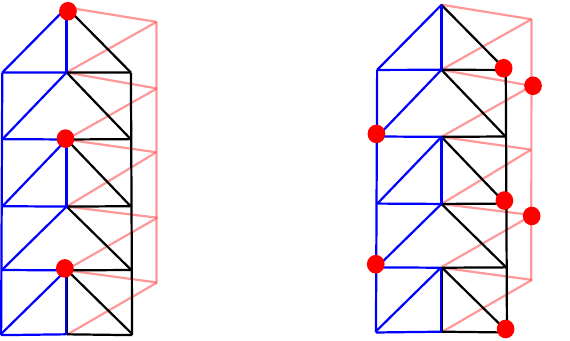}
\caption{Two configurations along an edge}
\label{edge_config}
\end{center}
\end{figure}

Figure \ref{total_config} describes a global example of a divisor $L \in H^2(\check X, \Z_2)$ satisfying $D^2 + LD = 0$ for all $D \in H^2(\check X, \Z_2)$. We have depicted the graph formed by the edges of $\partial P$. Each triple of vertices corresponds to a $2$-face of $\partial P$ and this $2$-faces is either an ``empty'' face or an ``arrow'' depending on whether the decoration of the edges matches the corresponding symbols. 
\begin{figure}[!ht] 
\begin{center}
\includegraphics{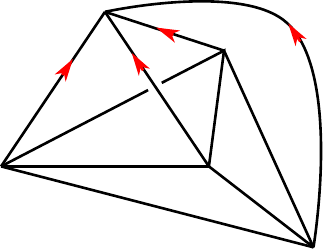}
\caption{A global configuration with empty faces and arrows}
\label{total_config}
\end{center}
\end{figure}
 It is clear that all non-empty edges will satisfy the configuration on the left in Figure \ref{edge_config} and all empty edges will satisfy the configuration on the right. 

\subsection{The twisted real mirror quintic.} Let $\check X$ be the mirror of the quintic. We study the topology of a twisted real mirror quintic in $\check X$, which we denote by $\check \Sigma_{\tau}$. In this case 
\[ H^1(B, R^{1}\check f_{\ast} \Z_2) \cong (\Z_2)^{101} \quad \text{and} \quad  H^1(B, R^{2}\check f_{\ast} \Z_2) \cong \Z_2. \]
In the untwisted case, $\check \Sigma$ has two connected components and it follows from the result of Arguz and Prince that $b_1(\check \Sigma_{0}) = 101$ (see Example 4.10 of \cite{arg_princ:realCY}). Since $H^1(B, R^{1} f_{\ast} \Z_2) \cong \Z_2$, there is only one twisted real mirror quintic $\check \Sigma_{\tau}$.  It follows from Lemma \ref{cohL2} and \ref{cohL1} that $H^1(B, \mathcal L^2_{\tau}) =0$ and  $H^1(B, \mathcal L^1_{\tau}) =(\Z_2)^{101}$. Then, sequence \eqref{connecting2} implies
\[ b_1(\Sigma_{\tau}) = 100. \]

\bibliographystyle{plain}

\vspace{1cm}
\begin{flushleft}
Diego MATESSI \\
Dipartimento di Matematica \\
Universit\`a degli Studi di Milano \\
Via Saldini 50 \\
I-20133 Milano, Italy \\
E-mail address: \email{diego.matessi@unimi.it}
\end{flushleft}

\end{document}